\documentclass[letterpaper,11pt]{article}
\usepackage{fullpage}
\usepackage[shortlabels]{enumitem}
\usepackage{enumitem}
\usepackage{amsmath,amsfonts,amssymb,amscd,mathtools}
\usepackage{amsthm}

\usepackage{xcolor}
\usepackage{graphicx}
\usepackage{bbm}
\usepackage{color,soul}
\usepackage{sidecap}
\usepackage{lineno,hyperref}

\theoremstyle{definition}

\newtheorem{proposition}{Proposition}

\numberwithin{equation}{section}

\usepackage{tikz}

\newcommand*\dd{\mathop{}\!\mathrm{d}}
\newcommand{\bx}{\boldsymbol{x}}

\newcommand{\bt}{{\boldsymbol{\theta}}}
\renewcommand{\dd}[1]{\ensuremath{\operatorname{d}\!{#1}}}

\newcommand{\norm}[1]{\left\Vert {#1}\right\Vert}

\newcommand{\cU}{\mathcal{U}}
\newcommand{\cM}{\mathcal{M}}

\newcommand{\bR}{\mathbb{R}}

\newcommand{\bfw}{\boldsymbol{w}}
\newcommand{\bftheta}{\boldsymbol{\theta}}

\newcommand{\bfx}{\boldsymbol{x}}
\newcommand{\Ucal}{\mathcal{U}}
\newcommand{\rdto}{r^{\text{DtO}}}

\newcommand{\rotd}{r^{\text{OtD}}}
\newcommand{\sfP}{\mathsf{P}}
\newcommand{\ip}[1]{\left\langle {#1}\right\rangle}

\newcommand{\bfP}{\boldsymbol{P}}

\newcommand{\bfeta}{\boldsymbol{\eta}}
\newcommand{\bfJ}{\boldsymbol{J}}
\newcommand{\Mcal}{\mathcal{M}}

\newcommand{\bfalpha}{\boldsymbol{\alpha}}
\newcommand{\bfbeta}{\boldsymbol{\beta}}

\title{
Sequential-in-time training of nonlinear parametrizations for solving time-dependent partial differential equations\thanks{Y.C., B.P.~were supported by the Office of Naval Research under award N00014-22-1-2728 and B.P.~additionally by the National Science Foundation under Grant No.~2046521.}
}
\author{Huan Zhang\thanks{Courant Institute of Mathematical Sciences, New York University, NY 
  (Corresponding author: pehersto@cims.nyu.edu)}
\and Yifan Chen\footnotemark[2]
\and Eric Vanden-Eijnden\footnotemark[2]
\and Benjamin Peherstorfer\footnotemark[2]}
\date{April 2024}

\newenvironment{keywords}%
   {\begin{trivlist}\item[]{\bfseries\sffamily Keywords:}\ }
   {\end{trivlist}}

\ifpdf
\hypersetup{
  pdftitle={Sequential-in-time training of nonlinear parametrizations for solving time-dependent partial differential equations}, 
  pdfauthor={Huan Zhang and Yifan Chen and Eric Vanden-Eijnden and Benjamin Peherstorfer} 
}
\fi
          
\begin{document}

\maketitle
\begin{abstract}
Sequential-in-time methods solve a sequence of training problems to fit nonlinear parametrizations such as neural networks to approximate solution trajectories of partial differential equations over time. 
This work shows that sequential-in-time training methods can be understood broadly as either optimize-then-discretize (OtD) or discretize-then-optimize (DtO) schemes, which are well known concepts in numerical analysis. The unifying perspective leads to novel stability and a posteriori error analysis results that provide insights into theoretical and numerical aspects that are inherent to either OtD or DtO schemes such as the tangent space collapse phenomenon, which is a form of over-fitting. 
Additionally, the unified perspective facilitates establishing connections between variants of sequential-in-time training methods, which is demonstrated by identifying natural gradient descent methods on energy functionals as OtD schemes applied to the corresponding gradient flows. 
\end{abstract}


\section{Introduction} We  introduce the setting of simulating time-dependent processes and systems that are given in the form of partial differential equations and discuss the need for nonlinearly parametrizing their solution fields. We continue with a literature review on sequential-in-time training methods for nonlinear parametrizations and sketch that these methods can be understood broadly as either discretize-then-optimize or optimize-then-discretize schemes, which outlines the paper.

\subsection{Simulating time-dependent processes and systems}
A core task of scientific computing and scientific machine learning is predicting the future behavior of time-dependent processes and systems. In many scenarios of interest, systems describe physical phenomena, where the goal of simulating them is to gain scientific insights or to solve engineering problems. A need to simulate time-dependent systems also arises in non-physical settings such as gradient flows corresponding to optimization and sampling in machine learning and other fields \cite{trillos2023optimization}. In any case, models of systems of interest are typically given in the form of time-dependent partial differential equations (PDEs). The task of simulating systems therefore becomes numerically solving PDEs, which often means approximating the solution fields with a parametrization (e.g., linear combination of basis functions, neural networks) that depends on a finite number of parameters (e.g., the coefficients of the linear combination, weights of network) and solving---training---for the parameters via algebraic equations \cite{FEMTheory,RotheDeuflhard}. 

\subsection{Limitations of linear parametrizations} The parametrization of the solution field critically influences the accuracy of the simulation result. We distinguish between linear and nonlinear parametrizations in the following. A parametrization is linear if the dependence on the parameters is linear, even though the dependence on spatial coordinates, time, coefficients, and other inputs of the PDEs can be nonlinear. Widely used linear parametrizations in scientific computing are linear combinations of basis functions. The basis functions are either local and centered at fixed grid points in the spatial domains of PDEs \cite{FEMTheory,RotheDeuflhard} or global as in spectral methods \cite{Gottlieb1977,BERNARDI1997209,Hesthaven_Gottlieb_Gottlieb_2007} and model reduction \cite{RozzaPateraSurvey,benner2015survey,interpbook,doi:10.1146/annurev-fluid-121021-025220}. Parametrizations based on linear combinations of basis functions have been shown to achieve fast error decays for a wide range of PDEs of interest; see \cite{MADAY2002289,10.1093/imanum/dru066} for details on the approximation-theoretic aspects of linear parametrizations.  
However, there are classes of PDEs for which linear parametrizations lead to slow error decays with respect to the number of parameters. One such class is given by PDEs over high-dimensional domains, for which grid-based methods can suffer from the curse of dimensionality \cite{bungartz_griebel_2004,Griebel_2006,Cohen_DeVore_2015}. Another class of problems for which linear approximations lead to slow error decays is given by PDEs that describe transport-dominated problems such as wave-like phenomena and strongly advecting flows  \cite{ROWLEY20001,Ohlberger16,Greif19,P22AMS}. For PDEs formulated over low-dimensional spatial domains, adaptive mesh refinement has been introduced to cope with transport; however, in settings such as kinetic equations over six-dimensional phase spaces and reduced models with global basis functions, mesh adaptation becomes either intractable or is not applicable and thus other parametrizations than linear ones are necessary. 

\subsection{Nonlinear parametrizations and global versus sequential-in-time training}\label{sec:Intro:Global} 
In this work, we consider parametrizations that depend nonlinearly on the parameters. Examples of nonlinear parametrizations are deep neural networks \cite{Goodfellow-et-al-2016}, tensor networks \cite{ORUS2014117,MAL-059}, and Gaussian wave packets \cite{BECK20001,lubich2008quantum}.
We distinguish between global-in-time and sequential-in-time methods for training nonlinear parametrizations. 
Methods that are global in time treat time as yet another variable, analogously to another spatial coordinate. Examples of global-in-time methods for nonlinear parametrizations are physics-informed neural networks \cite{RAISSI2019686}. Global-in-time methods are analogous to time-space discretizations in scientific computing. In contrast, sequential-in-time methods solve a sequence of problems to sequentially fit parameters so that the parametrized function approximates the PDE solution field over time. Thus, time is treated separately from the spatial coordinates and preserves its special meaning, which can be useful for preserving causality and keeping the number of parameters low \cite{doi:10.1137/050639703,BRUNA2024112588}. Methods that train sequentially in time share a close relationship with standard methods in scientific computing that discretize space first and then integrate forward in time a system of ordinary differential equations (ODEs).

\subsection{Literature review of sequential-in-time training methods for nonlinear parametrizations}\label{sec:Intro:LitReview}
There is a large body of literature on sequential-in-time training methods for nonlinear parametrizations. A key concept that is often employed is the Dirac-Frenkel variational principle 
\cite{dirac1930note, frenkel1934wave,Kramer1981}, which has been leveraged in particular by the computational chemistry community to numerically compute solutions of the Schr\"odinger equation with Gaussian wave packets \cite{MEYER199073,BECK20001,9073ba01-c8c8-3f30-b15c-e4b52a44e9da,lubich2008quantum}; we refer to \cite[Section~3.8]{lasser_lubich_2020} for a short history. Instead of formulating Galerkin conditions with respect to a test space that is fixed in time, the Dirac-Frenkel variational principle relies on the tangent space at the current solution of the manifold induced by the parametrization. The Dirac-Frenkel variational principle forms the foundation for dynamic low-rank approximations \cite{doi:10.1137/050639703,EINKEMMER2021110353,NobileDO,Peherstorfer15aDEIM,MUSHARBASH2018135,hesthaven_pagliantini_rozza_2022}, which have been extended to tensor formats \cite{doi:10.1137/09076578X,Arnold2014}. 
There is a range of works that apply the Dirac-Frenkel variational principle to other nonlinear parametrizations than Gaussian wave packets and matrix decompositions such as deep networks \cite{PhysRevE.104.045303,BRUNA2024112588,berman2023randomized,wen2023coupling,quantuminspired_NG,finzi2023a,kast2023positional} and nonlinear reduced approximations \cite{UngerTransformModes2020,anderson2021evolution,berman2024colora}. 
Closely related are particle methods \cite{doi:10.1146/annurev.fluid.37.061903.175753} that update functions represented as linear combinations of kernels over time; see also applications of meshless kernel methods for solving time dependent PDEs sequentially \cite{yan2023kernel}. The multiconfiguration time-dependent Hartree method propagates nonlinear parametrizations that are wavepackets over time \cite{MEYER199073,BECK20001}, which can be interpreted as tensor decompositions. For stochastic problems, dynamic orthogonal decomposition approaches have been proposed \cite{SAPSIS20092347,doi:10.1137/16M1109394,doi:10.1137/22M1534948,doi:10.1137/21M1431229}.
Other sequential-in-time methods for nonlinear parametrizations are proposed in the work  \cite{kvaal2023need} with Gaussian wave packets and in \cite{CHEN_implictnerual} with an implicit neural representation. A similar approach is taken by energetic variational methods \cite{wang2021particle,hu2023energetic}.   

\subsection{OtD and DtO schemes and summary of contributions}
The goal of this work is to show that there are two different broad types of sequential-in-time methods. There are optimize-then-discretize (OtD) schemes that first derive an optimization problem based on a parametrization to obtain a dynamical system in continuous time, which is then discretized in time and integrated forward. 
We show that schemes based on the Dirac-Frenkel variational principle are of the OtD type. The other type of methods are  discretize-then-optimize (DtO) schemes that first discretize in time and then solve a sequence of optimization problems, typically corresponding to boundary value problems, to sequentially fit the parameters. 
For example, the methods introduced in \cite{kvaal2023need,chen2023implicit} are DtO schemes.
Distinguishing between OtD and DtO  is analogous to distinguishing between the method of lines and the Rothe method in the context of linear parametrizations \cite{RotheDeuflhard}. The method of lines first discretizes in space and then numerically integrates the corresponding system of ODEs. In contrast, the Rothe method first discretizes in time to obtain a sequence of boundary value problems which are then numerically solved. Another analogy with linear approximations can be found in model reduction, where the authors of \cite{CARLBERG2017693} discuss least-squares Petrov-Galerkin methods that are DtO schemes and distinguish them to what they refer to as Galerkin methods that are OtD schemes. The same concept of OtD and DtO can be found in other settings such as inverse and control problems \cite{doi:10.1137/1.9780898718577,Becker2007} as well as in neural ordinary differential equations \cite{NEURIPS2018_69386f6b,ijcai2019p103,Onken2020DO} when fitting time series data.  

Distinguishing between DtO and OtD schemes enables a unified study of the theory and numerical aspects of sequential-in-time methods. 
We derive a posteriori error estimators as well as norm stability bounds for OtD and DtO schemes. The analysis gives insights into the properties of the two schemes: First, we will argue that OtD schemes have advantages in terms of implementation and numerical costs because if explicit time integration schemes are used, then it is sufficient to numerically solve linear least-squares problem at each time step. This is remarkable because the parametrization is nonlinear. However, OtD schemes can suffer from a phenomenon that we term tangent space collapse, which means that the residual is set orthogonal to a tangent space that looses expressiveness and thus residual components are ignored and can grow so that they lead to a deterioration of the solution accuracy. We also relate the tangent space collapse to the rank deficiency observed in dynamic low-rank approximations \cite{Lubich2014} and other methods based on the Dirac-Frenkel variational principle \cite{KAY1989165,10.1063/1.449204,PhysRevE.101.023313,feischl2024regularized}. 
We then provide an analysis of DtO schemes and derive a posteriori bounds that are unaffected by tangent space collapse phenomenon; however, this robustness of DtO schemes comes with typically higher computational costs because non-convex optimization problems have to be solved at each time step, instead of linear ones as in OtD schemes with explicit time integrators.

We also discuss how OtD and DtO schemes are related. We show that under strong assumptions, OtD schemes provide first-order approximations of solutions of DtO schemes. 
Besides problems stemming from physics applications that are modeled as PDEs, we consider PDEs that arise from gradient flows corresponding to optimization and sampling  problems \cite{trillos2023optimization}. In particular, we show that applying natural gradient descent to an energy functional over a parametric class is equivalent to applying an OtD scheme to the corresponding gradient flow equation. We thus can recover several algorithms based on natural gradient descent by applying variations of OtD schemes on the corresponding gradient flows, which offers insights for deriving a wider range of efficient algorithms.

\subsection{Outline of the paper}
In Section~\ref{sec: The DtO and OtD Framework}, we set the stage by introducing the PDE problems and the nonlinear parametrizations that we consider. Schemes of the OtD type are discussed in Section~\ref{sec:OtD}, including a posteriori error analysis and stability. We also discuss the importance of the tangent space in OtD schemes. The topic of Section~\ref{sec:DtO} are DtO schemes and their error analysis and stability. Section~\ref{sec:Discussion of OtD and DtO schemes} shows that only under very strong assumptions OtD and DtO schemes coincide. The section also connects OtD schemes on gradient flows to a range of sequential-in-time methods available in the literature. Conclusions are drawn in  Section~\ref{sec: Conclusions}.

\section{Nonlinear parametrizations for time-dependent PDEs}
\label{sec: The DtO and OtD Framework}
We discuss nonlinear parametrizations of solution fields of time-dependent PDEs.

\subsection{Setup}
\label{sec-Time-dependent nonlinear parametrizations of PDE solution fields}
Consider a time-dependent PDE over the spatial domain $\Omega\subseteq \mathbb{R}^d$:
\begin{equation}
\label{eq:Prelim:PDE}
\begin{aligned}
    \partial_t u(t,\bx) = f(t,\bx,u), \quad &\text{ for } (t,\bx) \in (0,T] \times \Omega\,,\\
    u(0,\bx)=u_0(\bx), \quad &\text{ for } \bx \in \Omega\, ,
\end{aligned}
\end{equation}
where the solution field is $u:[0,T]\times \Omega \to \mathbb{R}$ and the initial condition is $u_0: \Omega \to \mathbb{R}$. The right-hand side $f$ can contain partial derivatives of the function $u$. Time is denoted with $t$ and final time is $T$. 
In the following, we only consider situations where the equation \eqref{eq:Prelim:PDE} with appropriate boundary condition is well posed and admits a unique solution in $C^1([0,T],\cU)$, 
where $\cU$ is an appropriate Hilbert space of functions with domain $\Omega$ that can be embedded into $C^0(\Omega)$; thus functions in $\cU$ admit point-wise evaluations. 
The inner product on $\cU$ is denoted by $\langle\cdot,\cdot\rangle_{\Ucal}$, and $\|\cdot\|_{\Ucal}$ is the induced norm. We write the solution $u$ at time $t$ as $u(t, \cdot): \Omega \to \mathbb{R}$ to highlight that at a fixed time $t$, $u(t, \cdot)$ is a function of the spatial coordinate and is an element of  $\cU$. 
For simplicity of the discussion, we only consider Dirichlet boundary conditions.

\subsection{Nonlinear time-dependent parametrizations}\label{sec:Prelim:Param}
We parametrize the solution field $u$ as $\hat{u}(\bftheta(t),\cdot):\Omega \to \bR$, which depends on a finite-dimensional parameter $\bftheta(t) \in \Theta \subseteq \mathbb{R}^{p}$ that can vary with time $t$. For simplicity, we only consider the case where the boundary conditions that accompany \eqref{eq:Prelim:PDE} are imposed directly by the parametrization so that any function $\hat{u}(\bftheta, \cdot)$ with $\bftheta \in \Theta$ satisfies the boundary conditions. We further assume that $\hat{u}$ is sufficiently regular in both of its arguments. In particular, we only consider parametrizations with $\hat{u}(\bt, \cdot) \in \cU$ for all $\bt \in \Theta$ and $\|f(t,\cdot,\hat{u}(\bftheta,\cdot))\|_{\cU} < \infty$ for any $\bftheta \in \Theta$ and $t \in [0,T]$.

We highlight two properties of the parametrization $\hat{u}$: First, the parametrization $\hat{u}$ can depend nonlinearly on the parameter $\bftheta(t)$. For example, the parametrization can be given by deep neural networks \cite{Goodfellow-et-al-2016} and tensor networks \cite{ORUS2014117,MAL-059} with time-dependent parameter vector $\bftheta(t)$. The nonlinearity of such parametrizations is a key distinguishing feature compared to traditional parametrizations used in numerical analysis where the parameter $\bftheta(t)$ enters linearly (e.g., coefficients of linear combinations of basis functions that are centered at grid points). The nonlinear dependence on the parameter that we adopt here can be interpreted as adapting the representation of the solution field.  
Second, the parameter $\bftheta$ is a function of time $t$, which is in contrast to a wide range of global-in-time (or time-space) approaches that build on nonlinear parametrizations where time $t$ enters as an input but the parameters are fixed over time; see Section~\ref{sec:Intro:Global}. 
Because the parametrizations have time-dependent parameter vectors $\bftheta(t)$, it is necessary to determine the evolution of $\bftheta(t)$ such that $\hat{u}(\bftheta(t), \cdot)$ solves \eqref{eq:Prelim:PDE} in some numerical sense.

\section{Optimize-then-Discretize (OtD) schemes}
\label{sec:OtD}
In this section, we discuss OtD schemes to solve for the time-dependent parameter $\bftheta(t)$. 
Schemes based on OtD are analogous to the method of lines \cite{FEMTheory,RotheDeuflhard}, which first derives a semi-discrete system of ODEs 
and then discretizes and numerically integrates the system of ODEs in a second step. 
In case of nonlinear parametrizations, we formulate OtD schemes that first define a residual function based on the PDE and the nonlinear parametrization, which leads to an optimization problem that depends continuously on time $t$. 
The first-order optimality conditions determined by the residual objective over time $t$ can then be written as a dynamical system in $\bftheta(t)$, which subsequently is discretized and integrated forward in time. 

\subsection{Description of OtD schemes} We now describe OtD schemes. 
\subsubsection{Residual function in OtD schemes}
Plugging $\hat{u}(\bftheta(t),\cdot)$ into the PDE \eqref{eq:Prelim:PDE} and applying the chain rule leads to the OtD residual function
\begin{equation}\label{eq:Schemes:OtD:ResFun}
    \rotd(t, \bftheta(t),\dot{\bftheta}(t),\cdot)  = \nabla_\bt \hat{u}(\bftheta(t),\cdot)^T \dot{\bftheta}(t)-f(t,\cdot,\hat{u}(\bftheta(t), \cdot)) \,,
    \end{equation}
defined over the spatial domain $\Omega$. 
Here $\nabla_{\bftheta}\hat{u}(\bftheta(t), \cdot): \Omega \to \mathbb{R}^p$ is the gradient of $\hat{u}$ with respect to the parameter $\bftheta$ with component functions $\nabla_{\theta_i}\hat{u}(\bftheta(t), \cdot): \Omega \to \mathbb{R}$ for $i = 1, \dots, p$. 
Consider now the optimization problem  
\begin{equation}\label{eq:Schemes:OtD:RegProb:short}
\min_{\boldsymbol{\eta} \in \Theta} \left\|\rotd (t,  \bftheta(t), \boldsymbol{\eta}, \cdot)\right\|_{M}^2\, ,
\end{equation}
with the norm of the residual as objective. 
In \eqref{eq:Schemes:OtD:RegProb:short}, we have the auxiliary variable $\boldsymbol{\eta}$. The norm $\|\cdot\|_M$ is defined on $\cU$ with its corresponding inner product $\langle\cdot,\cdot\rangle_M$. It can be different from $\|\cdot\|_{\cU}$ but we need that $\norm{f(t,\cdot,\hat{u}(\bt,\cdot))}_M<\infty$ for all $t\in [0,T]$ and over all $\bt\in \Theta$; see Section~\ref{sec:Prelim:Param}.  

\subsubsection{Optimality conditions in OtD schemes} Taking the gradient of the objective of \eqref{eq:Schemes:OtD:RegProb:short} with respect to $\bfeta$ and setting it to zero determines a first-order optimal point of \eqref{eq:Schemes:OtD:RegProb:short}, which we identify as $\dot{\bftheta}(t)$ and which satisfy
\begin{equation}\label{eq:Schemes:OtD:FirstOrderOpti}
\langle\nabla_{\theta_i} \hat{u}(\bt(t), \cdot), \rotd(t, \bt(t), \dot{\bt}(t), \cdot))\rangle_M=0\,,\qquad i = 1, \dots, p\, .
\end{equation}
Interpreting the inner product to apply component-wise, we write equation \eqref{eq:Schemes:OtD:FirstOrderOpti} in the form of dynamics for $\bftheta(t)$ as
\begin{equation}\label{eq:Schemes:OtD:FirstOrderOpti_NG}
\left\langle\nabla_\bt \hat{u}(\bt(t), \cdot), \nabla_\bt \hat{u}(\bt(t), \cdot)  \right\rangle_M\dot{\bt}(t)=\left\langle\nabla_\bt \hat{u}(\bt(t), \cdot), f(t,\cdot,\hat{u}(\bt(t), \cdot)) \right\rangle_M\,.
\end{equation}
We define
\begin{equation}\label{eq:OtD:MatrixP}
    \bfP(\bftheta) = \langle\nabla_\bt \hat{u}(\bt, \cdot), \nabla_\bt \hat{u}(\bt, \cdot)\rangle_M\,,\quad
    \boldsymbol{F}(\bftheta) = \langle\nabla_\bt \hat{u}(\bt(t), \cdot), f(t,\cdot,\hat{u}(\bt(t), \cdot))\rangle_M\,,
\end{equation}to compactly write \eqref{eq:Schemes:OtD:FirstOrderOpti_NG} as
\begin{align}
    \bfP(\bftheta(t))\dot{\bt}(t)=\boldsymbol{F}(\bftheta(t))\,,
\label{eq:Schemes:OtD:FirstOrderOpti_NG_compact}
\end{align} which is a system of ODEs if the matrix $\bfP(\bftheta(t))$ is non-singular. 
Otherwise, system \eqref{eq:Schemes:OtD:FirstOrderOpti_NG_compact} can include differential-algebraic equations that attain multiple solutions; using the Moore–Penrose inverse of $\bfP(\bftheta)$ is one way to uniquely determine a system of ODEs. In either case, we assume the resulting ODEs \eqref{eq:Schemes:OtD:FirstOrderOpti_NG} are well-posed such that the solution $\bftheta(t) \in C^1([0,T])$; then $\hat{u}(\bftheta(t),\cdot)$ is also continuously differentiable in time because $\hat{u}$ is continuously differentiable in the parameter argument. 

System \eqref{eq:Schemes:OtD:FirstOrderOpti_NG} describes the normal equations of the least-squares problem
\begin{equation}\label{eq:Schemes:OtD:RegProb}
\operatorname*{min}_{\bfeta \in \Theta} \|\nabla_{\bftheta}\hat{u}(\bftheta(t), \cdot)^T \bfeta - f(t, \cdot, \hat{u}(\bftheta(t), \cdot))\|_M^2\,.
\end{equation}
It is numerically advantageous to directly solve the least-squares problem \eqref{eq:Schemes:OtD:RegProb} rather than the poorer conditioned normal equations \eqref{eq:Schemes:OtD:FirstOrderOpti_NG_compact}.   
We stress that the term inside the norm in \eqref{eq:Schemes:OtD:RegProb} is linear in the unknown $\bfeta$, even though the parameter $\bftheta(t)$ enters nonlinearly in the parametrization $\hat{u}$. 

A major challenge of OtD (and also of DtO) schemes is the numerical approximation of the inner product $\langle \cdot, \cdot \rangle_{M}$ in \eqref{eq:Schemes:OtD:FirstOrderOpti_NG} or equivalently the evaluation of the norm $\|\cdot\|_M$ in \eqref{eq:Schemes:OtD:RegProb}.
For certain specific combinations of nonlinear parametrizations and PDEs, such as Gaussian wave packets and variants of the Schr\"odinger equation \cite{lasser_lubich_2020}, the objective of \eqref{eq:Schemes:OtD:RegProb} can be computed analytically. For more general nonlinear parametrizations such as deep networks, it has been proposed to use quadrature rules \cite{PhysRevE.104.045303,berman2023randomized}, which can work well if the spatial domain $\Omega$ is of low dimensions. For higher dimensions,  
adaptive Monte Carlo methods have been developed~\cite{BRUNA2024112588,wen2023coupling}. 

\subsubsection{OtD schemes and the Dirac-Frenkel variational principle}
Since $\partial_t \hat{u}(\bftheta(t), \cdot) = \nabla_\bt \hat{u}(\bftheta(t),\cdot)^T \dot{\bftheta}(t)$, we can derive from \eqref{eq:Schemes:OtD:FirstOrderOpti_NG} that
\begin{equation}
\label{eqn-3.7}
\begin{aligned}
    \left\langle\nabla_\bt \hat{u}(\bt(t), \cdot), \partial_t \hat{u}(\bftheta(t), \cdot)  \right\rangle_M&=\left\langle\nabla_\bt \hat{u}(\bt(t), \cdot), f(t,\cdot,\hat{u}(\bt(t), \cdot)) \right\rangle_M \\
    &= \left\langle\nabla_\bt \hat{u}(\bt(t), \cdot), \sfP_{\bt} f(t,\cdot,\hat{u}(\bt(t), \cdot)) \right\rangle_M \,,
\end{aligned}
\end{equation}
where $\sfP_{\bt}$ is the projection operator onto the tangent space $T_{\hat{u}(\bt, \cdot)}\cM$ at $\hat{u}(\bftheta, \cdot)$ of a manifold induced by the parametrization $\cM = \{\hat{u}(\bftheta, \cdot) \,|\, \bftheta \in \Theta\}$; the projection is  
defined under the $\langle\cdot,\cdot\rangle_M$ inner product.  The tangent space $T_{\hat{u}(\bt, \cdot)}\cM$ at $\hat{u}(\bftheta, \cdot)$ is spanned by the component functions of the gradient $\nabla_{\bftheta}\hat{u}(\bftheta, \cdot)$. Equation \eqref{eqn-3.7} implies that the term $\partial_t \hat{u}(\bftheta(t), \cdot) - \sfP_{\bt} f(t,\cdot,\hat{u}(\bt(t), \cdot)) \in T_{\hat{u}(\bt, \cdot)}\cM$ is orthogonal to the tangent space $T_{\hat{u}(\bt, \cdot)}\cM$ and thus that the difference must be zero  because $\partial_t \hat{u}(\bftheta(t), \cdot) \in T_{\hat{u}(\bt, \cdot)}\cM$. We obtain the following evolution equation in the function space:
\begin{align}
    \partial_t \hat{u}(\bftheta(t), \cdot) = \sfP_{\bt(t)}f(t,\cdot,\hat{u}(\bftheta(t), \cdot))\, .
    \label{eq:projected_dynamics}
\end{align}
\begin{SCfigure}
\begin{tikzpicture}
\tikzset{x={(175:1cm)},y={(55:.7cm)},z={(90:1cm)}}
\shade[left color=blue!4,right color=blue!60,looseness=.5, draw=black]  (2.5,-2.5,-1) 
to[bend left] (2.5,2.5,-1)
to[bend left] coordinate (mp) (-2.5,2.5,-1)
to[bend right] (-2.5,-2.5,-1)
to[bend right] coordinate (mm) (2.5,-2.5,-1)
-- cycle;
\shade[opacity=0.4, draw=black] (2.5,-2.5,0) -- (2.5,2.5,0) -- (-2.5,2.5,0) -- (-2.5,-2.5,0) -- cycle;
\draw[-stealth, thick] (-0.5,0.08,0) -- coordinate[pos=.3] (f) (1.0,1.0,1);
\draw[-stealth] (-0.5,0.08,0) -- coordinate[pos=.3] (f) (1.0,1.0,0);
\draw[dotted] (1.0,1.0,1) -- (1.0,1.0,0);
\draw[solid]  (1.0,1.0,0) -- (1.0,1.0,0.25) -- (0.75,0.83,0.25) -- (0.75,0.83,0);
\filldraw (-0.5,0.08,0) circle (1pt);
\node[right, scale=0.8] at (-0.1,-0.5,0) {{$\hat{u}(\bftheta(t), \cdot)$}};
\node[right, scale=0.8] at (1.7,-0.3,0.2) {{$\partial_t\hat{u}(\bftheta(t), \cdot)$}};
\node[above, scale=0.8] at (1.0,1.0,1.0) {{$f(\cdot, \hat{u}(\bftheta(t),\cdot))$}};
\draw (-1.5,2.5,0) node[above, rotate=-4] {\small{tangent space at $\hat{u}(\bftheta(t),\cdot)$}}; 
\draw (0,-4.1,0) node[right] {$\mathcal{M}$};
\end{tikzpicture}
	\caption{The Dirac--Frenkel variational principle \cite{dirac1930note, frenkel1934wave,Kramer1981,lubich2008quantum} determines the time derivative $\dot\bftheta(t)$ of the parameter $\bftheta(t)$ via the orthogonal projection of the right-hand side $f( t, \cdot, \hat{u}(\bftheta(t), \cdot))$ onto the tangent space of the manifold $\cM$ at the current solution field $\hat{u}(\bftheta(t), \cdot)$. The tangent space is spanned by the component functions of the gradient $\nabla_{\bftheta}\hat{u}(\bftheta(t), \cdot)$.}
	\label{fig:manifold_illustration_plain}
\end{SCfigure}
Through the lens of \eqref{eq:projected_dynamics}, the optimization step in OtD schemes can be interpreted using the Dirac-Frenkel principle; see Figure~\ref{fig:manifold_illustration_plain} and the discussion in Section~\ref{sec:Intro:LitReview}.  

Equation \eqref{eq:projected_dynamics} plays a key role in the error and stability analysis of the continuous OtD dynamics in the subsequent sections. It is noteworthy that while the well-definedness of \eqref{eq:Schemes:OtD:FirstOrderOpti_NG} in the parameter space may incur necessary extra discussions when $\boldsymbol{P}(\bt)$ is singular (i.e., we need to pick a specific solution of the linear system and ensure this choice leads to a well-posed ODE) and different dynamics of $\bftheta(t)$ may arise, equation \eqref{eq:projected_dynamics} remains consistently well-defined in the function space and has the same formula given any of these dynamics of $\bftheta(t)$. 
Therefore, in the continuous OtD dynamics, whether $\boldsymbol{P}(\bt(t))$ is singular or not will not affect the form of the equation in the function space and a posteriori error and stability analysis remain the same; see Section~\ref{sec-A posteriori error analysis of OtD schemes}--\ref{sec:OtD:Stability}.

\subsubsection{Discretization in time}\label{sec:OtD:DiscTime}
We set the time-step size to $\delta t > 0$ and denote the time steps as $k \in \mathbb{N}$ corresponding to times $t_k = k \delta t$ with the time-discrete approximations $\bftheta_0, \bftheta_1, \bftheta_2, \dots$ of the time-continuous parameter functions $\bftheta(0), \bftheta(t_1), \bftheta(t_2), \dots$. The time-continuous least-squares problem \eqref{eq:Schemes:OtD:RegProb} (or equivalently its normal equations  \eqref{eq:Schemes:OtD:FirstOrderOpti_NG}) 
can be discretized with off-the-shelf time integrators. For example, we can discretize in time with the $\zeta$-scheme, where $\zeta \in [0, 1]$. 

Setting $\zeta = 1$ corresponds to an explicit Euler time discretization of \eqref{eq:Schemes:OtD:RegProb} as
\begin{equation}\label{eq:OtD_explicitReg}
\operatorname*{min}_{\bftheta_{k + 1} \in \Theta} \|\nabla_{\bftheta}\hat{u}(\bftheta_k, \cdot)^T (\bftheta_{k + 1} - \bftheta_k) - \delta t f(t_k, \cdot, \hat{u}(\bftheta_k, \cdot))\|_M^2\,,\qquad k \in \mathbb{N}\,,
\end{equation}
and of the normal equations \eqref{eq:Schemes:OtD:FirstOrderOpti_NG} as
\begin{align}
    \left\langle\nabla_\bt \hat{u}(\bt_{k},\cdot), \nabla_\bt \hat{u}( \bt_{k},\cdot)\right\rangle_M (\bt_{k+1}-\bt_{k})=\delta t\left\langle\nabla_\bt \hat{u}(\bt_{k},\cdot),f(t_{k},\cdot, \hat{u}(\bt_{k},\cdot))\right\rangle_M \,.
    \label{eq:OtD_explicit}
\end{align}
For $\zeta = 0$, the implicit Euler discretization leads to
\begin{equation}\label{eq:OtD_implicit}
\operatorname*{min}_{\bftheta_{k + 1} \in \Theta} \|\nabla_{\bftheta}\hat{u}(\bftheta_{k + 1}, \cdot)^T (\bftheta_{k + 1} - \bftheta_k) - \delta t f(t_{k+1}, \cdot, \hat{u}(\bftheta_{k + 1}, \cdot))\|_M^2\,,\quad k \in \mathbb{N}\,,
\end{equation}
and to analogous equations for the normal equations \eqref{eq:Schemes:OtD:FirstOrderOpti_NG}.  
In the time-continuous formulation via regression problem \eqref{eq:Schemes:OtD:RegProb} (and the normal equations \eqref{eq:Schemes:OtD:FirstOrderOpti_NG}), the unknown variable $\boldsymbol{\eta}$  enters linearly. When discretizing with an explicit scheme as in \eqref{eq:OtD_explicitReg}, the time-discrete system is also linear in the unknown $\bftheta_{k + 1}$. In contrast, because $\hat{u}$ depends nonlinearly on the parameter $\bftheta$, the regression problem \eqref{eq:OtD_implicit} corresponding to the implicit Euler discretization is nonlinear in $\bftheta_{k + 1}$ since $\bftheta_{k + 1}$ enters in the parametrization $\hat{u}$. 

\subsection{A posteriori analysis of OtD schemes}
\label{sec-A posteriori error analysis of OtD schemes}
In this section, we analyse the error  of OtD solutions. 
The main purpose of our analysis is demonstrating that the accuracy of OtD solutions critically depends on the tangent spaces, for which we can build on results from the literature \cite{9073ba01-c8c8-3f30-b15c-e4b52a44e9da,lubich2008quantum,lasser_lubich_2020} but go beyond by developing error and stability results that apply in more general settings. 

In Proposition \ref{eq:OtD:APosterioriBound}, we analyze the case where $f(t,\cdot,u)$ is Lipschitz in $u$. The proof follows standard arguments for deriving error bounds in solving ODEs; for example as in \cite{9073ba01-c8c8-3f30-b15c-e4b52a44e9da} in the context of the Dirac-Frenkel variational principle. The norm in which the error is measured can be chosen generically and is denoted by $\|\cdot\|$ with a corresponding inner product $\langle \cdot, \cdot \rangle$. 

\begin{proposition}{(See \cite{9073ba01-c8c8-3f30-b15c-e4b52a44e9da}.)}\label{eq:OtD:APosterioriBound} Consider the time-dependent PDE \eqref{eq:Prelim:PDE} and let $\bftheta(t)$ solve the continuous OtD dynamics \eqref{eq:projected_dynamics} so that $\hat{u}(\bt(t),\cdot)$ approximates $u$.  
Assume that there exists a non-negative constant $C$ such that for all $t \in [0,T]$ and $v_1,v_2 \in \cU$,
            \begin{align}
                \norm{f(t,\cdot,v_1)-f(t,\cdot,v_2)}\leq C\norm{v_1-v_2}\,.
                \label{eq:lipschitzcf}
            \end{align} 
Furthermore, assume that there exists a function $\varepsilon: [0,T]\to [0,\infty)$ so that 
\begin{equation}\label{eq:OtDAposteriori:FTangentSpace}
    \norm{f(t, \cdot, \hat{u}(\bt(t), \cdot))-\sfP_{\bftheta(t)} f(t, \cdot, \hat{u}(\bt(t), \cdot))}\leq \varepsilon(t)\, .
\end{equation}
Then, the following error bound holds:         \begin{align}\label{eq:OtDerrorbound}
            \norm{u(t, \cdot)-\hat{u}(\bt(t),\cdot)}\leq \textrm{e}^{Ct}\norm{u(0, \cdot)-\hat{u}(\bt(0),\cdot)}+\textrm{e}^{Ct}\int_0^t \textrm{e}^{-Cs}\varepsilon(s){\rm d}s\, .
        \end{align}
\end{proposition}
\begin{proof} The proof follows similar arguments as  \cite{9073ba01-c8c8-3f30-b15c-e4b52a44e9da}: 
        By the Cauchy-Schwartz inequality, property \eqref{eq:lipschitzcf}, and dynamics \eqref{eq:Prelim:PDE} and \eqref{eq:projected_dynamics},
        \begin{align}
            \begin{split}
            \frac{\dd{}}{\dd{t}}\|u(t, \cdot)-&\hat{u}(\bt(t),\cdot)\|^2= 2\left\langle u(t, \cdot)-\hat{u}(\bt(t),\cdot), \partial_t (u(t, \cdot)-\hat{u}(\bt(t),\cdot) )\right\rangle\\
            & = 2\left\langle u(t, \cdot)-\hat{u}(\bt(t),\cdot), f(t,\cdot,u(t,\cdot))-f(t, \cdot, \hat{u}(\bt(t), \cdot))\right\rangle+\\&\qquad2\left\langle u(t, \cdot)-\hat{u}(\bt(t),\cdot),f(t, \cdot, \hat{u}(\bt(t), \cdot))-\sfP_{\bftheta(t)} f(t, \cdot, \hat{u}(\bt(t), \cdot)) \right\rangle\\
            &\leq 2C\norm{u(t, \cdot)-\hat{u}(\bt(t),\cdot)}^2 + 2\varepsilon(t)\norm{u(t, \cdot)-\hat{u}(\bt(t),\cdot)}.
             \end{split}
        \end{align} 
        Now, dividing by $\norm{u(t, \cdot)-\hat{u}(\bt(t),\cdot)}$ on both sides, we have
        \begin{align}
        \label{eqn-3-17}
            \frac{\dd{}}{\dd{t}}\norm{u(t, \cdot)-\hat{u}(\bt(t),\cdot)}&\leq C\norm{u(t, \cdot)-\hat{u}(\bt(t),\cdot)} + \varepsilon(t).
        \end{align}
     Applying Gr\"ownwall's inequality leads to the desired result \eqref{eq:OtDerrorbound}. Note that if the term $\norm{u(t, \cdot) - \hat{u}(\bt(t),\cdot)}$ approaches zero at certain times so that the norm becomes not differentiable, we can integrate \eqref{eqn-3-17} from the last point where zero occurs and extend it forward and so the bound that we obtain is still valid. 
\end{proof}

The error bound \eqref{eq:OtDerrorbound} depends on $\epsilon(t)$, which bounds the projection error of the right-hand side onto the tangent space of the nonlinear parametrization. 
Thus, the a posteriori analysis shows that the right-hand side of the PDE needs to be well approximated by the tangent space of the nonlinear parametrization at the current solution function $\hat{u}(\bftheta(t), \cdot)$; see Section~\ref{sec:OtD:TangentSpace} for details. 
In the special case where $\epsilon(t)=0$, meaning that the right-hand side function lies in the tangent space, the time evolution will not introduce additional errors beyond those in the initial condition, which is leveraged in, e.g., \cite{lasser_lubich_2020}. 

In the follow proposition, we consider the case where $f$ includes unbounded differential operators such as the Laplacian operator, which violates the Lipschitz assumption on $f$ that   Proposition~\ref{eq:OtD:APosterioriBound} relies on.  
We note that the smallest non-zero eigenvalue of $-\Delta$ with Dirichlet boundary conditions, denoted by $\lambda^*$, is defined as $\min_{u \in \mathcal{V}, u \neq 0} \frac{\|\nabla u\|_{L^2(\Omega)}}{\|u\|_{L^2(\Omega)}}$, where $\mathcal{V} = H_0^1(\Omega)$ \cite{evans2022partial}.

\begin{proposition}\label{eq:Prop:OtD:Laplace}
Consider the time-dependent PDE \eqref{eq:Prelim:PDE} with homogeneous Dirichlet boundary conditions. We assume the solution space $\cU$ embeds into the Sobolev space $H^2(\Omega)$ and that the right-hand side $f$ of \eqref{eq:Prelim:PDE} has the form 
$f(t, \cdot, u(t, \cdot))$ $= \Delta u(t, \cdot) + g(t, \cdot, u(t, \cdot)) $ and 
the function $g$ satisfies
\begin{align}
     \norm{g(t,\cdot,v_1)-g(t,\cdot,v_2)}_{L^2(\Omega)}\leq C\norm{v_1-v_2}_{L^2(\Omega)}\,,
    \label{eq:lipschitzcg}
\end{align} 
 with a non-negative constant $C$ for all $t \in [0,T]$ and $v_1,v_2 \in \cU$. Furthermore, let $\lambda^*>0$ be the smallest non-zero eigenvalue of $-\Delta$ over $\mathcal{V} = H_0^1(\Omega)$. Let now $\hat{u}(\bftheta(t), \cdot)$ solve the continuous OtD dynamics \eqref{eq:projected_dynamics} with $\langle\cdot,\cdot\rangle_M = \langle\cdot,\cdot\rangle_{L^2(\Omega)}$  
and let there exists a function $\varepsilon: [0,T]\to [0,\infty)$ so that 
\begin{equation}\label{eq:OtDAPosteriori:BoundGTangent}
    \norm{f(t, \cdot, \hat{u}(\bt(t), \cdot))-\sfP_{\bftheta(t)} f(\cdot, \hat{u}(\bt(t), \cdot))}_{L^2(\Omega)}\leq \varepsilon(t)\, .
\end{equation}
Then, the following error bound holds with constant $C_1 = C - \lambda^*$:
\begin{equation}
\label{eq:OtDerrorbound_laplacian}
\begin{aligned}
    \norm{u(t, \cdot)-\hat{u}(\bt(t),\cdot)}_{L^2(\Omega)}\leq \mathrm{e}^{C_1t}\left(\norm{u(0, \cdot)-\hat{u}(\bt(0),\cdot)}_{L^2(\Omega)}
    +\int_0^t \mathrm{e}^{-C_1s}\varepsilon(s){\rm d}s\right).
\end{aligned}
\end{equation}
\end{proposition}
\begin{proof}
By direct calculations, we have
        \begin{align}
            \begin{split}
            &\quad\frac{\dd{}}{\dd{t}}\norm{u(t, \cdot)-\hat{u}(\bt(t),\cdot)}^2_{L^2(\Omega)}\\
            &= 2\left\langle u(t, \cdot)-\hat{u}(\bt(t),\cdot), \partial_t (u(t, \cdot)-\hat{u}(\bt(t),\cdot) )\right\rangle_{L^2(\Omega)}\\& =
            2\left\langle u(t, \cdot)-\hat{u}(\bt(t),\cdot), \Delta (u(t, \cdot)-\hat{u}(\bt(t),\cdot) )\right\rangle_{L^2(\Omega)}\\
            &
            \qquad +2\left\langle u(t, \cdot)-\hat{u}(\bt(t),\cdot), g(t,\cdot,u(t,\cdot))-g(t, \cdot, \hat{u}(\bt(t), \cdot))\right\rangle_{L^2(\Omega)}\\
            &\qquad +2\left\langle u(t, \cdot)-\hat{u}(\bt(t),\cdot),f(t, \cdot, \hat{u}(\bt(t), \cdot))-\sfP_{\bftheta(t)} f(t, \cdot, \hat{u}(\bt(t), \cdot)) \right\rangle_{L^2(\Omega)}\label{eq:replacable_with_linear}\\
            &\leq -2\lambda^* \norm{u(t, \cdot)-\hat{u}(\bt(t),\cdot)}^2_{L^2(\Omega)}+2C\norm{u(t, \cdot)-\hat{u}(\bt(t),\cdot)}^2_{L^2(\Omega)} + \\&\qquad 2\varepsilon(t)\norm{u(t, \cdot)-\hat{u}(\bt(t),\cdot)}_{L^2(\Omega)}\,.
             \end{split}
        \end{align} 
        where in the last inequality we have used the spectral property of the Laplacian operator: $\langle u(t, \cdot)-\hat{u}(\bt(t),\cdot), \Delta (u(t, \cdot)-\hat{u}(\bt(t),\cdot))\rangle_{L^2(\Omega)}\leq -\lambda^* \norm{u(t, \cdot)-\hat{u}(\bt(t),\cdot)}^2_{L^2(\Omega)}$ with $\lambda^*>0$. 
        Therefore, using the same argument as in the proof of Proposition \ref{eq:OtD:APosterioriBound}, we obtain
        \begin{align*}
            \frac{\dd{}}{\dd{t}}\norm{u(t, \cdot)-\hat{u}(\bt(t),\cdot)}_{L^2(\Omega)}
            &\leq (C-\lambda^*) \norm{u(t, \cdot)-\hat{u}(\bt(t),\cdot)}_{L^2(\Omega)}+\varepsilon(t),
        \end{align*}
        with which we can derive \eqref{eq:OtDerrorbound_laplacian} via Gr\"ownwall's inequality.
\end{proof}

The error bound \eqref{eq:OtDerrorbound_laplacian} depends on the accumulation of the projection error bounded by $\epsilon(t)$, analogous to the bound obtained in Proposition~\ref{eq:OtD:APosterioriBound}. 
If an asymmetric linear operator $\mathcal{A}$ can be isolated from either $f$ or $g$, a tighter error bound can be derived. Specifically, observing that $\left\langle v, \mathcal{A} v\right\rangle=-\left\langle \mathcal{A} v, v\right\rangle=0$ holds, the asymmetric part does not contribute to the growth of $\norm{u(t, \cdot)-\hat{u}(\bt(t),\cdot)}^2_{L^2(\Omega)}$. As an illustrative example, such an operator may arise from an advection term $\mathcal{A}v = \boldsymbol{b}\cdot \nabla v$ where $\boldsymbol{b}: \Omega\to \mathbb{R}^{d}$ is a divergence-free drift. This is evident as $\left\langle v, \boldsymbol{b}\cdot \nabla v\right\rangle =-\left\langle \nabla \cdot (\boldsymbol{b}v ), v\right\rangle=-\left\langle (\nabla \cdot\boldsymbol{b}) v, v\right\rangle-\left\langle \boldsymbol{b}\cdot\nabla v, v\right\rangle
=-\left\langle \boldsymbol{b}\cdot\nabla v, v\right\rangle$.

\subsection{The importance of tangent spaces}\label{sec:OtD:TangentSpace} 
The error analysis provided by Proposition~\ref{eq:OtD:APosterioriBound} and \ref{eq:Prop:OtD:Laplace} shows that the accuracy of the OtD solutions is critically influenced by the accumulation of the  projection error of the right-hand side onto the tangent spaces. 

\subsubsection{Choice of parametrization} The importance of the tangent space can be used to inform the choice of the nonlinear parametrization. For example, when solving semi-classical Schr\"odinger equations with a Gaussian wave packets ansatz, then the tangent spaces exactly represent the right-hand side functions as long as the potential used in the Schr\"odinger equation is quadratic \cite{lasser_lubich_2020}. Another example is given by the linear advection equation with right-hand side $f(t,\cdot, u(t,\cdot)) = \nabla u(t,\cdot)^T\bfw$, where $\bfw \in \mathbb{R}^d$. Consider $\bftheta = [\bfalpha; \bfbeta]$ with $\bfalpha\in \mathbb{R}^{d}, \bfbeta\in \mathbb{R}^{p_1}$ that define the parametrization $\hat{u}(\bt, \bfx)=\Phi(\bfx+\bfalpha; \bfbeta)$ with a function $\Phi: \Omega \times \mathbb{R}^{p_1} \to \mathbb{R}$. In this case, the right-hand side is $f(t,\cdot, \hat{u}(\bt,\cdot))= \nabla_{\bfx} \hat{u}(\bt, \cdot)^T\bfw=\nabla_{\bfalpha} \hat{u}(\bt,\cdot)^T\bfw$, which is in the tangent space $T_{\hat{u}(\bt, \cdot)}\Mcal$ because it is a linear combination of the component functions of $\nabla_{\bftheta}\hat{u}(\bftheta, \cdot)$. Thus, for the linear advection equation and the given parametrization, time-continuous OtD schemes incur no error over time.

\subsubsection{Collapsing tangent spaces}\label{sec:OtD:Collapsing} Even if $\hat{u}(\bftheta(t), \cdot)$ is a good approximation of $u(t, \cdot)$, it does not necessarily imply that the tangent space at $\hat{u}(\bftheta(t), \cdot)$ is expressive to approximate well the right-hand side via the projection $\sfP_{\bftheta(t)}$; see equation~\eqref{eq:projected_dynamics}. In fact, the component functions of the gradient $\nabla_{\bftheta} \hat{u}(\bftheta(t),\cdot)$ can become linearly dependent, at least numerically. Numerically linearly dependent means that the component functions are close to being  linearly dependent when the inner product between them is numerically estimated with Monte Carlo sampling or some other quadrature method. Equivalently, the matrix $\bfP(\bftheta)$ can become numerically singular or poorly conditioned \cite{KAY1989165,PhysRevE.101.023313,10.1063/1.449204,Lubich2014,kvaal2023need,berman2023randomized}. We refer to this phenomenon loosely as collapsing tangent space phenomenon. 
It is common that nonlinear parametrizations lead to gradients with (at least numerically) linearly dependent component functions, which can be related to the neural co-adaptation phenomena and overfitting in deep network approximations \cite{HintonDropout,srivastavaDropout,berman2023randomized}. This has major implications because it means that even though at time $t$ one obtains an accurate solution, the accuracy cannot be maintained as time evolves due to the collapse of the tangent space. 

The dynamics in the function space as formulated in \eqref{eq:projected_dynamics} remain well-posed even if the component functions of the gradient are linearly dependent; however, the loss of accuracy still applies. In terms of the dynamics formulated over the parameter $\bftheta(t)$ as in \eqref{eq:Schemes:OtD:FirstOrderOpti_NG}, the matrix $\bfP(\bt)$  can be singular due to collapsing tangent spaces and the Moore-Penrose inverse can be used to force a unique trajectory $\bftheta(t)$; again, which does not hinder the loss of accuracy even though the dynamics are well posed. For numerical time integration, small singular values of $\bfP(\bt)$ can still pose challenges. For dynamic low-rank approximations, where the parametrizations is a matrix or tensor decomposition with time-dependent factors, robust time integrators based on projector-splitting have been proposed \cite{Lubich2014,10.1093/amrx/abv006,doi:10.1137/18M116383X} for near-singular $\bfP(\bftheta)$. The projector-splitting can be applied to the spatially discretized system \cite{10.1093/amrx/abv006} and the continuous system \cite{doi:10.1137/18M116383X}. The work \cite{feischl2024regularized} proposes a minimal-norm regularization scheme to cope with the tangent space collapse.

\subsubsection{OtD dynamics with collapsed tangent spaces}\label{sec:OtD:TangentCollapseExample}
Let us consider a concrete example, where we use the nonlinear parametrization
\[
\hat{u}(\bt, x)=\sum\nolimits_{i=1}^N \beta_i \phi(x-\alpha_i)\,,\qquad x \in \mathbb{R}\,,
\]
with $\bt = [\alpha_1,...,\alpha_N, \beta_1, ..., \beta_N] \in \mathbb{R}^{2N}$ denoting the concatenation of all parameters. The nonlinear function $\phi: \mathbb{R} \to \mathbb{R}$ can be a Gaussian kernel function, for example. Consider now the scenario where, at time $t_0$, the parameter $\bftheta(t_0)$ of the  numerical approximation $\hat{u}(\bt(t_0), \cdot)$ satisfies $\alpha_i (t_0)=\alpha_j(t_0)$ and  $\beta_i(t_0)=\beta_j(t_0)$ for all $i,j = 1, \dots, N$. Such a parameter $\bftheta(t_0)$ means that there are at most two linearly independent components functions of $\nabla_{\bt} \hat{u}$ at  $\bt(t_0)$. 

As discussed in Section \ref{sec:OtD:Collapsing}, when the tangent space collapses, there are arbitrarily many parameter vectors $\bftheta(t)$ that solve the linear system corresponding to OtD schemes, necessitating the selection of a particular parameter vector. If the singular value decomposition is used to solve the regression problem \eqref{eq:Schemes:OtD:RegProb}, then the minimal Euclidean norm parameter vector is selected. In our particular case where $\nabla_{\alpha_i} \hat{u}(\bt(t), \cdot)) = \nabla_{\alpha_j} \hat{u}(\bt(t), \cdot)$ and 
$\nabla_{\beta_i} \hat{u}(\bt(t), \cdot)) = \nabla_{\beta_j} \hat{u}(\bt(t), \cdot)$ for all $i, j = 1, \dots, N$, the $i$-th and $j$-th and $i+N$-th and $j+N$-th columns of the matrix $\bfP(\bftheta(t))$ 
are the same, respectively. Therefore, the first term in the objective of \eqref{eq:Schemes:OtD:RegProb}, namely $\nabla_\bt \hat{u}(\bt(t), \cdot)^T \dot{\bftheta}(t)$, remains unchanged as long as $\sum_i \dot{\alpha}_i(t)$ and $\sum_i \dot{\beta}_i(t)$ remain constant. 
As a consequence, the solutions $\dot{\bftheta}(t)$ of \eqref{eq:Schemes:OtD:RegProb} lie on the manifold with $\sum_i \dot{\alpha}_i(t) = c_{\alpha}$, $\sum_i \dot{\beta}_i(t) = c_{\beta}$ for some constants $c_{\alpha}, c_{\beta}$. 
The minimal norm parameter vector in the Euclidean norm is the one with $\dot{\alpha}_i(t) = \dot{\alpha}_j(t)$ and $\dot{\beta}_i(t) = \dot{\beta}_j(t)$.  
Consequently, $\alpha_i(t) = \alpha_j(t)$ and $\beta_i(t) = \beta_j(t)$ will continue to hold throughout the evolution over time $t$, maintaining the degeneracy. The matrix $\bfP(\bt(t))$ will remain singular in the subsequent steps and thus the rank of the tangent space is not increased over time. In other words, the tangent space cannot get more expressive. Effectively, only one basis function contributes to the expressiveness of the nonlinear parametrization. As a result, the accuracy can be unsatisfactory.

\subsubsection{Avoiding tangent space collapse in OtD schemes} While the OtD dynamics remain valid for collapsing tangent spaces, it is still desired to avoid the collapse to maintain accurate solutions over time. To see this, first notice that  a tangent space collapse can drive solutions into parameter regions that cannot be easily escaped from anymore, as shown in the example in Section~\ref{sec:OtD:TangentCollapseExample}. 
One option to escape in such a situation is to artificially increase the rank as discussed in  \cite{Lubich2014,Ceruti2022,doi:10.1137/22M1534948}, in the context of projector-splitting and other robust time integration schemes for dynamic low-rank approximations. 

Another option that builds on randomized updates is proposed in \cite{berman2023randomized}: 
Recall that the parameter trajectory $\bftheta(t)$ is numerically computed via time integration schemes that take steps of the size $\delta t$ at each time step. Such a time integration leads to local movements in the parameter domain $\Theta$ and it thus can take many steps to escape poor regions. In particular, selecting a specific regularizer (e.g., minimal norm solutions by using the Moore-Penrose pseudo inverse of $\bfP(\bftheta)$) to enforce non-singular dynamics in the parameter $\bftheta(t)$  can drive $\bftheta(t)$ into regions that take a long time to escape from. An analogous issue of escaping local regions with only local steps is found in Markov chain Monte Carlo methods, where it can take a large number of steps before the low probability regions between two metastable states is traversed \cite{doi:10.1137/21M1425062}. To allow global steps, the work  \cite{berman2023randomized} proposes randomized sparse OtD schemes. These schemes update random sparse subsets of the components of $\bftheta(t)$ at each time step. The randomization of which components of $\bftheta(t)$ are updated can be interpreted as allowing global steps, which is empirically shown in \cite{berman2023randomized} to alleviate the collapsing tangent space phenomenon and the poor conditioning problem to some extent. 

Finally, we mention here already that DtO schemes avoid the tangent space collapsing phenomena in favor of typically higher optimization costs; see Section~\ref{sec:DtO}. We remark that discretizing OtD schemes in time with implicit time integrators can help to alleviate the tangent space collapsing phenomenon too: As shown in the objective of \eqref{eq:OtD_implicit}, the tangent space at the parameter $\bftheta_{k + 1}$ at the next time step $k + 1$ is used to represent the right-hand side. Thus, optimizing the objective \eqref{eq:OtD_implicit} seeks $\hat{u}(\bftheta_{k + 1}, \cdot)$ so that the corresponding tangent space is expressive for representing the right-hand side, which avoids relying on the tangent space at the current solution $\hat{u}(\bftheta_k, \cdot)$. However, implicit time discretizations of OtD schemes require solving non-convex optimization problems at each time step (see Section~\ref{sec:OtD:DiscTime}) and thus loose the major benefit of OtD schemes of typically incurring lower computational costs per time step compared to DtO schemes. At the same time, the implicit discretization avoids an issue that the time-continuous formulation is affected by, which is poised to lead to inconsistencies in the limit of smaller time-step sizes.

\subsection{Stability of continuous OtD dynamics}\label{sec:OtD:Stability}
In this section, we analyze the stability of the OtD dynamics, more specifically the growth of $\hat{u}(\bt(t), \cdot)$ under certain norms. 
The stability bounds are independent of a posteriori terms that are not readily available a priori such as the projection error bound $\epsilon(t)$ that is used in the error analysis provided in Section \ref{sec-A posteriori error analysis of OtD schemes}. 

\subsubsection{Stability when right-hand sides are bounded}\label{sec:OtD:StabBounded}
We now show bounds on the norm growth for problems with bounded right-hand sides $f$.  
\begin{proposition}\label{prop:OtDAnalysis:NormStab} Let $\hat{u}(\bt(t), \cdot)$ satisfy the continuous OtD dynamics \eqref{eq:projected_dynamics}. 
Assume there exist constants $C, C_0 > 0$ such that 
\begin{equation}\label{eq:OtDAnalysis:NormStab:FBound}
\norm{f(t, \cdot, \hat{u}(\bt(t),\cdot)}_M\leq C \norm{\hat{u}(\bt(t),\cdot)}_M+C_0\,.
\end{equation}
Then it holds that
\begin{equation}\label{eq:OtDAnalysis:BoundNormStab}
    \norm{\hat{u}(\bt(t),\cdot)}_M\leq \norm{\hat{u}(\bt(0),\cdot)}_M \mathrm{e}^{Ct}+\frac{C_0}{C}\big(\mathrm e^{Ct}-1\big)\,.
\end{equation}
\end{proposition}
\begin{proof}
Because $\sfP_{\bt(t)}$ is a projection with respect to $\langle \cdot, \cdot \rangle_M$, it holds $\norm{\sfP_{\bt(t)} w}_M\leq \norm{w}_M$.  
Using \eqref{eq:projected_dynamics}, we get
\begin{align*}
    \begin{split}
    \frac{\dd{}}{\dd{t}}\frac{1}{2}\norm{\hat{u}(\bt(t), \cdot)}_M^2&=\ip{\hat{u}(\bt(t), \cdot), \partial_t \hat{u}(\bt(t), \cdot)}_M=\ip{\hat{u}(\bt(t), \cdot),\sfP_{\bt(t)} f(t,\cdot, \hat{u}(\bt, \cdot))}_M\\&=\ip{\sfP_{\bt(t)}\hat{u}(\bt(t), \cdot),  f(t,\cdot, \hat{u}(\bt(t), \cdot))}_M\\&\leq \norm{\sfP_{\bt(t)} \hat{u}(\bt(t), \cdot)}_M\norm{ f(t,\cdot, \hat{u}(\bt(t), \cdot))}_M\\&\leq  C\norm{\hat{u}(\bt(t), \cdot)}_M^2+ C_0  \norm{\hat{u}(\bt(t),\cdot)}_M\,.
    \end{split}
\end{align*}
As a consequence, we obtain
\begin{align}
    \frac{\dd{}}{\dd{t}}\norm{\hat{u}(\bt(t), \cdot)}_M&\leq  C\norm{\hat{u}(\bt(t), \cdot)}_M+ C_0 \,.
\end{align}
Applying Gr\"ownwall's lemma leads to \eqref{eq:OtDAnalysis:BoundNormStab}. 
\end{proof}

When $f(t,\cdot,u)=(\mathcal{A}u) (\cdot)$, where $\mathcal{A}$ is an anti-symmetric linear operator such that $\ip{w, \mathcal{A}v}_M=-\ip{\mathcal{A}w, v}_M$, then the norm $\norm{\hat{u}(\bt(t),\cdot)}_M$ does not grow. 
The preservation of norm has been demonstrated in the solution of the semi-classical Schrödinger equation with Gaussian wave packets~\cite{lasser_lubich_2020}, where $\mathcal{A}=-i H$ and $H$ represents a Hamiltonian operator.

\subsubsection{Stability for right-hand sides with unbounded operators}\label{sec:OtD:StabUnBounded}
The following proposition builds on an assumption on the nonlinear parametrization that once more emphasizes the importance of the tangent spaces:
\begin{equation}
    \label{eq:Asm:UInT}
    \hat{u}(\bt, \cdot) \in T_{\hat{u}(\bt, \cdot)}\cM\,,\qquad \text{ for all } \bt \in \Theta\,.
\end{equation}
Assumption \eqref{eq:Asm:UInT} holds for a wide range of parametrizations. For example, in \cite[Section 3]{lasser_lubich_2020}, it is shown that \eqref{eq:Asm:UInT} holds for parametrizations based on Gaussian wave packets. More generally, the assumption given in \eqref{eq:Asm:UInT} holds for feed-forward deep networks as long as the last layer is linear, in which case the network can be written as $\hat{u}(\bt, \cdot) = \sum_{i = 1}^N w_i \phi_i(\cdot; \bt') + b_i$, 
where $\bt = [\bt', w_1, \dots, w_N, b_1, \dots, b_N] \in \Theta$ is the parameter vector and the functions $\phi_1, \dots, \phi_N: \Omega \to \mathbb{R}$ correspond to inner layers. Note that the network can have multiple layers, which are encoded in $\phi_1, \dots, \phi_N$. By differentiating the network with respect to $\bt$ one can see that $\phi_1, \dots, \phi_N$ and the constant function are component functions of the gradient $\nabla_{\bt}\hat{u}$ and thus \eqref{eq:Asm:UInT} holds.

\begin{proposition}\label{prop:OtDAnalysis:NormStabLaplace}
Let $\hat{u}(\bt(t), \cdot)$ satisfy the continuous OtD dynamics \eqref{eq:projected_dynamics} with $\langle \cdot, \cdot \rangle_{M} = \langle \cdot, \cdot \rangle_{L^2(\Omega)}$ and homogeneous Dirichlet boundary condition. Assume further the parametrization satisfies \eqref{eq:Asm:UInT} and $\hat{u}(\bt(t), \cdot) \in H^2(\Omega)$. The right-hand side of \eqref{eq:Prelim:PDE} has the form $f(t, \cdot, u(t, \cdot)) = \Delta u(t, \cdot) + g(t, \cdot, u(t, \cdot))$ and there exist constants $C, C_0 > 0$ 
 such that $
 \|g(t, \cdot, \hat{u}(\bt(t), \cdot))\|_{L^2(\Omega)}\leq C \|\hat{u}(\bt(t), \cdot)\|_{L^2(\Omega)}+C_0\,.$ 
Then, 
\begin{equation}\label{eq:OtDAnalysis:BoundNormWithLaplace}
    \norm{\hat{u}(\bt(t),\cdot)}_{L^2(\Omega)}\leq \norm{\hat{u}(\bt(0),\cdot)}_{L^2(\Omega)} \textrm{e}^{(C-\lambda^*)t}+\frac{C_0}{C-\lambda^*}\big(\textrm{e}^{(C-\lambda^*)t}-1\big)\,,
\end{equation}
holds with $\lambda^*>0$ being the smallest non-zero eigenvalue of $-\Delta$ with Dirichlet boundary conditions.
\end{proposition}
\begin{proof}
With a parametrization that satisfies \eqref{eq:Asm:UInT}, it holds $\hat{u}(\bt, \cdot)=\sfP_{\bt} \hat{u}(\bt, \cdot)$. Using the same notation as in the proof of Proposition~\ref{prop:OtDAnalysis:NormStab}, we obtain
    \begin{align*}
        \begin{split}
        \frac{\dd{}}{\dd{t}}\frac{1}{2}\|\hat{u}(\bt(t), \cdot)&\|^2_{L^2(\Omega)}=\ip{\hat{u}(\bt(t), \cdot), \partial_t \hat{u}(\bt(t), \cdot)}_{L^2(\Omega)}\\
        &=\ip{\hat{u}(\bt(t), \cdot),\sfP_{\bt(t)} f(t,\cdot, \hat{u}(\bt, \cdot))}_{L^2(\Omega)}\\
        &=\ip{ \hat{u}(\bt(t), \cdot), f(t,\cdot, \hat{u}(\bt, \cdot))}_{L^2(\Omega)}\\
        &=\ip{ \hat{u}(\bt(t), \cdot),  \Delta \hat{u}(\bt, \cdot) }_{L^2(\Omega)}+\ip{ \hat{u}(\bt(t), \cdot), g(t,\cdot, \hat{u}(\bt, \cdot))}_{L^2(\Omega)}\\
        &\leq -\lambda^* \norm{\hat{u}(\bt(t), \cdot)}^2_{L^2(\Omega)}+C\norm{\hat{u}(\bt(t), \cdot)}^2_{L^2(\Omega)}+ C_0  \norm{\hat{u}(\bt(t),\cdot)}_{L^2(\Omega)},
        \end{split}
    \end{align*}
where we used in the last step the spectral property of the Laplacian operator as in the proof of Proposition~\ref{eq:Prop:OtD:Laplace}. 
We obtain
\begin{align}
    \frac{\dd{}}{\dd{t}}\norm{\hat{u}(\bt(t), \cdot)}_{L^2(\Omega)}\leq (C -\lambda^* )\norm{\hat{u}(\bt(t), \cdot)}_{L^2(\Omega)}+ C_0\, ,
\end{align}
to which we apply Gr\"ownwall's inequality for arriving at \eqref{eq:OtDAnalysis:BoundNormWithLaplace}. 
\end{proof}

\subsection{Remark on the analysis of time-discretization OtD schemes}
\label{sec:discretization}
We focused the error and stability analyses of OtD schemes on continuous dynamics.
In the case of discretized OtD schemes, one may further apply standard error and stability analysis theories for ODE discretizations \cite{HairerOne}. More specifically, if we employ Euler discretizations with a time-step size $\delta t$ to solve \eqref{eq:Schemes:OtD:FirstOrderOpti_NG}, 
then the time-discrete approximations $\bftheta_k$ of $\bftheta(t_k)$ ($t_k = k\delta t$) is of order $O(\delta t)$, assuming the second-order derivative $\ddot{\bftheta}(t)$ is bounded. According to the formulation of the ODE, the boundedness of $ \ddot{\bftheta}(t)$ is related to the behavior of the Hessian $\nabla^2_{\bftheta}\hat{u}(\bftheta(t),\cdot)$. This means the curvature of the manifold $\cM$ of the nonlinear parametrization matters. Once a bound in the $\bftheta$-space is obtained, one may transfer it into the function space, namely $\|\hat{u}(\bftheta(t_k),\cdot) - \hat{u}(\bftheta_k,\cdot)\|$, using a bound on $\nabla_{\bftheta}\hat{u}(\bftheta,\cdot)$. 

\section{Discretize-then-Optimize (DtO) schemes} 
\label{sec:DtO}
In this section, we discuss DtO schemes to solve for the parameter $\bftheta(t)$. 
Schemes based on DtO first discretize in time and then optimize for the time-discrete parameter. The analogous approach with linear parametrizations is the Rothe method \cite{Rothe1930,RotheDeuflhard}, which first discretizes time and then obtains a sequence of boundary value problems over function spaces that are then numerically solved. 

\subsection{Description of the DtO schemes}
We now describe DtO schemes.

\subsubsection{Discretization in time in DtO schemes}
Let $\delta t > 0$ be the time-step size for applying the $\zeta$-scheme to \eqref{eq:Prelim:PDE}, which leads to
\begin{equation}\label{eq:Schemes:DtO:DiscTime}
\frac{u_{k+1}(\bx)  - u_{k}(\bx) }{\delta t} = \zeta f(t_k, \bx, u_k) + (1 - \zeta) f(t_{k + 1}, \bx,  u_{k + 1}),\  \bx \in \Omega\,, k \in \mathbb{N}\,,
\end{equation}
where $u_{k} \in \cU$ approximates $u(t_k, \cdot)$. The initial condition determines $u_0$. Notice that \eqref{eq:Schemes:DtO:DiscTime} describes a sequence of boundary value problems, instead of an initial value problem as the original problem \eqref{eq:Prelim:PDE}. Discretizing with $\zeta = 0$ and $\zeta = 1$ corresponds to the implicit and explicit Euler method, respectively. 

\subsubsection{First-order optimality conditions and DtO schemes}
Now consider the parametrization $\hat{u}$ and plug it into \eqref{eq:Schemes:DtO:DiscTime} to obtain the residual function
\begin{equation}\label{eq:Schemes:DtO:ResFun}
\rdto_k(\bftheta, \cdot) = \hat{u}(\bftheta, \cdot) - \big(\hat{u}(\bftheta_k,\cdot) + \delta t \zeta f(t_k,\cdot, \hat{u}( \bftheta_k, \cdot)) + \delta t(1 - \zeta) f(t_{k + 1},\cdot, \hat{u}( \bftheta,\cdot))\big)\, .
\end{equation}
Notice that $\rdto_k$ is independent of a time derivative of the parameter, in contrast to the residual function \eqref{eq:Schemes:OtD:ResFun} of OtD schemes. 
Given $\bftheta_k$ from the previous time step $k$, a parameter at the next time step $k + 1$ can be obtained via the optimization problem
\begin{equation}\label{eq:Schemes:DtO:RegProb}
\min_{\bftheta \in \mathbb{R}^p} \left\|\rdto_k( \bftheta,\cdot)\right\|_M^2\, .
\end{equation}
Calculations show that first-order optimal points $\bftheta_{k + 1}$ satisfy
\begin{equation}
\label{eq:Schemes:DtO:FirstOrderOpti}
\left\langle\nabla_{\bftheta}\rdto_k( \bftheta_{k + 1},\cdot) , \rdto_k(\bftheta_{k + 1},\cdot)\right\rangle_M  =0\,,
\end{equation}
which can be interpreted as testing the residual at the component functions of the gradient of the residual $\rdto_k$. Notice the difference to OtD schemes and their corresponding first-order optimality conditions \eqref{eq:Schemes:OtD:FirstOrderOpti}, where the residual is tested against the component functions of the gradient of the parametrization $\hat{u}$.
In case of an explicit Euler time discretization ($\zeta = 1$ in \eqref{eq:Schemes:DtO:DiscTime}),  condition \eqref{eq:Schemes:DtO:FirstOrderOpti} admits the form
\begin{equation}\label{eq:Schemes:DtO:FirstOrderOpti:zeta_1}
\left\langle \nabla_\bt \hat{u}(\bt_{k + 1},\cdot), r_k^{\text{DtO}}(\bt_{k + 1},\cdot)\right\rangle_M  = 0,\  \qquad k \in \mathbb{N}\,,
\end{equation}
because the gradient of the residual becomes the gradient of the parametrization at $\bt_{k + 1}$. However, equation \eqref{eq:Schemes:DtO:FirstOrderOpti:zeta_1} is still different from the first-order optimality conditions of OtD schemes because the residual is tested at the gradient at the subsequent time step $k + 1$ rather than at $k$. Nevertheless, condition \eqref{eq:Schemes:DtO:FirstOrderOpti:zeta_1} will be useful later to establish the stability of the stationary point trajectory of DtO schemes; see Section~\ref{sec:stability analysis of DtO}. 

Let us add the remark that there are analogous schemes in model reduction with linear parametrizations, which lead to Galerkin versus least-squares Petrov Galerkin schemes \cite{CARLBERG2017693}. A loose analogy is that OtD schemes correspond to what the authors of \cite{CARLBERG2017693} call Galerkin schemes and DtO schemes to the least-squares Petrov-Galerkin schemes, except with the major difference that OtD and DtO schemes considered in this work here apply to nonlinear parametrizations whereas in classical model reduction only linear parametrizations are considered.

\subsubsection{Gauss-Newton method for DtO schemes}
As the dependence of $\hat{u}$ on $\bftheta$ is nonlinear, numerically solving \eqref{eq:Schemes:DtO:RegProb} can be challenging. Even if an explicit time integration scheme is used, one still has to solve a nonlinear (and non-convex) optimization problem at each time step; in contrast to OtD schemes. While DtO schemes are not suffering from collapsing tangent space phenomena in the sense that a posteriori error bounds in the following section are independent of the projection errors onto the tangent spaces, the optimization step is typically more challenging than in OtD schemes and insufficient optimization in DtO schemes can lead to similarly poor accuracy as collapsing tangent spaces in OtD schemes (see Section~\ref{sec-Connections and Differences Between OtD and DtO}). 

Of particular interest to us is applying the Gauss-Newton method \cite{WrightNumOpt} to the optimization problem \eqref{eq:Schemes:DtO:RegProb}. For applying Gauss-Newton optimization, we define the matrix function
\begin{equation}\label{eq:DtO:PrecondJ}
\bfJ_k(\bftheta) = \langle\nabla_\bt \rdto_k( \bt,\cdot), \nabla_\bt \rdto_k( \bt,\cdot)\rangle_M\,.
\end{equation}
Applying the Gauss-Newton method to \eqref{eq:Schemes:DtO:RegProb} then leads to the linear system
\begin{equation}\label{eq:Compute:DtO:GaussNewton}
\bfJ_k(\bftheta_{k + 1}^{(l)})\bftheta_{k + 1}^{(l + 1)} = \bfJ_k(\bftheta_{k + 1}^{(l)})\bftheta_{k + 1}^{(l)} - \alpha\langle \nabla_\bt \rdto_k(\bt_{k+1}^{(l)},\cdot), \rdto_k(\bt_{k+1}^{(l)},\cdot) \rangle_M, \qquad l \in \mathbb{N}\,,
\end{equation}
at time step $k$ with step size $\alpha$, where we denote the optimization iterations with $l \in \mathbb{N}$ and the intermediate approximations of the parameters as $\bftheta_k^{(l)}$. At $l = 0$, we set $\bftheta_{k + 1}^{(0)} = \bftheta_k$ and after $L$ iterations we stop the optimization iterations and set $\bftheta_{k+1} = \bftheta_{k+1}^{(L)}$.

We stress that other optimization methods than the Gauss-Newton method can be applied to the DtO regression problem \eqref{eq:Schemes:DtO:RegProb}. Following standard practice in machine learning, stochastic gradient descent and its variants provide viable options.

\subsection{A posteriori error analysis of DtO schemes}
\label{sec: Analysis of DtO}
We now provide an a posteriori error analysis for solutions obtained with DtO schemes. 

\begin{proposition}
\label{prop-DtO-error-Lipschitz}
Let the right-hand side $f$ of \eqref{eq:Prelim:PDE} be Lipschitz continuous in $u$ so that there exists a constant $C > 0$ such that 
\begin{equation}\label{eq:Lipschitzness_of_f}
        \norm{f(t,\cdot,v)-f(t,\cdot,w)}\leq C\norm{v-w}\,,
\end{equation}
for $v, w \in \cU$ and all times $t \in [0,T]$. Consider the explicit Euler discretization of \eqref{eq:Prelim:PDE} with time-step size $\delta t$. Let $e_k: \Omega \to \mathbb{R}$ denote the time-integration error at time step $k$ such that
\begin{equation}\label{eq:definition_ek}
e_k(\cdot)=u(t_{k + 1}, \cdot) - u(t_k, \cdot)-\delta t f(t, \cdot, u(t_k, \cdot)) ,\  \quad k \in \mathbb{N}\,.
\end{equation}
Let $\hat{u}(\bt_k, \cdot)$ be the solution obtained in the DtO scheme with the residual function $r^{\rm DtO}_k (\bt_{k+1},\cdot)$. 
Then at time $t_k$, the error of the DtO solution can be bounded as
\begin{multline}\label{eq:DtO:APostUnBoundBound}
\norm{u(t_k,\cdot)-\hat{u}(\bt_{k}, \cdot)} \leq (1+C\delta t)^k \norm{u(0,\cdot)-\hat{u}(\bt_{0}, \cdot)}+\\\sum\nolimits_{i=0}^{k-1} (1+C\delta t)^{k-i-1}   (\norm{e_i(\cdot)}+\|r^{\rm DtO}_i(\bt_{i+1},\cdot)\|)\,.
\end{multline}

\end{proposition}
\begin{proof}
   By definition of $e_k(\cdot)$ in \eqref{eq:definition_ek} and $r_k^{\text{DtO}}$ in \eqref{eq:Schemes:DtO:ResFun}, we have
   \begin{equation}
       \begin{aligned}
    &\|u(t_{k+1},\cdot) - \hat{u}(\bt_{k}, \cdot)\| \\
    = &\|u(t_k,\cdot) -\hat{u}(\bt_{k}, \cdot) + \delta t f(t_k, \cdot, u(t_k, \cdot))-\delta t f(t_k, \cdot, \hat{u}(\bt_{k}, \cdot))  -r_k^{\text{DtO}}(\bt_{k+1},\cdot)+e_{k}\|\\
    \leq & \|u(t_k,\cdot) -\hat{u}(\bt_{k}, \cdot) \| + \delta t\|f(t_k, \cdot, u(t_k, \cdot))- f(t_k, \cdot, \hat{u}(\bt_{k}, \cdot))\| + \|r_k^{\text{DtO}}(\bt_{k+1},\cdot)-e_{k}\|\\
     \leq & (1+C\delta t) \norm{u(t_k,\cdot)-\hat{u}(\bt_{k}, \cdot)}+\norm{r_k^{\text{DtO}}(\bt_{k+1},\cdot)}+\norm{e_k(\cdot)},
       \end{aligned}
   \end{equation}
   where we used the Lipschitzness of $f (t, \cdot, u(t_k, \cdot))$  and the triangle inequality. 
    A telescoping sum yields \eqref{eq:DtO:APostUnBoundBound}. 
\end{proof}

The above proposition shows that if $f$ is Lipschitz in $u$ then the error of the DtO solution is bounded by a sum of errors in the initial condition, the time discretization error of the PDE, and the norm of the residual over the spatial domain at each time step. In particular, the error of solutions of DtO schemes does not directly depend on the projections onto the tangent spaces of $\cM$ of the nonlinear parametrization, in contrast to OtD schemes.
When taking $\|\cdot\|=\|\cdot\|_M$, one can make  $\norm{\rdto_k(\bftheta, \cdot)}_M$ small through sufficient optimization in the optimization step of DtO and using expressive nonlinear parametrizations. Overall, we summarize that while the main emphasis in OtD schemes is on the expressiveness of the tangent space comprising component functions of $\nabla_{\bftheta} \hat{u}(\bftheta,\cdot)$ to ensure the projection of the right-hand side is accurate, in DtO schemes one focuses on the expressiveness of the nonlinear parametrization $\hat{u}(\bftheta,\cdot)$ itself as well as sufficient optimization so that one can drive the norm of the residual and therefore also the error low.

Proposition \ref{prop-DtO-error-Lipschitz} assumes $f$ is Lipschitz. When $f$ contains unbounded differential operators such as the Laplacian operator, DtO schemes with explicit time integrators may not necessarily lead to bounds as above. As an example, consider the case $f(t,\cdot, u(t,\cdot))=\Delta u(t, \cdot)+g(t,\cdot, u(t,\cdot))$ and  $\|\cdot\|_M = \|\cdot\|_{L^2(\Omega)}$ in the DtO scheme. Suppose the explicit Euler discretization is used, then selecting $\bftheta_{k + 1}$ that minimizes the residual norm $\|\rdto_k(\bftheta_{k + 1}, \cdot)\|_{L^2(\Omega)}$ means that the error of $\hat{u}(\bftheta_{k+1},\cdot)$ in the $L^2(\Omega)$ norm is low, under the assumption that $\Delta \hat{u} (\bftheta_{k},\cdot)$ and $\hat{u}(\bt_{k}, \cdot)$ have low errors, due to the following relation
\[\hat{u}(\bt_{k+1}, \cdot)=\hat{u}(\bt_{k}, \cdot) + r_{k}^{\text{DtO}}(\bftheta_{k + 1}, \cdot) + \delta t \Delta \hat{u} (\bftheta_{k},\cdot) + \delta t g(t_k, \cdot, \hat{u}(\bt_{k}, \cdot)).\]
However, the DtO optimization objective can only guarantee $\|\rdto_k(\bftheta_{k + 1}, \cdot)\|_{L^2(\Omega)}$ is small, but not the $H^2(\Omega)$ norm of it, so we do not know whether $\Delta \hat{u} (\bftheta_{k+1},\cdot)$ is accurate or not. Thus, in the next iteration, which takes the form
\[\hat{u}(\bt_{k+2}, \cdot)=\hat{u}(\bt_{k+1}, \cdot) + r_{k+1}^{\text{DtO}}(\bftheta_{k + 2}, \cdot)+ \delta t \Delta \hat{u} (\bftheta_{k+1},\cdot) + \delta t g(t_{k+1}, \cdot, \hat{u}(\bt_{k+1}, \cdot)),\]
minimizing the $L^2(\Omega)$ norm of the residual will be insufficient to keep the error of $\hat{u}(\bt_{k+2}, \cdot)$ low.

To obtain bounded operators at each time step, one can use different norms in the loss; however, these can be numerically challenging to evaluate. Instead, we use implicit time integration in DtO schemes to obtain the  following proposition.
 
\begin{proposition}
\label{prop-DtO-implicit-error}
 Consider the time-dependent PDE \eqref{eq:Prelim:PDE} with a homogeneous Dirichlet boundary condition, where the right-hand side of the PDE has the form
\begin{align}
    f(t,\cdot, u(t,\cdot))=\Delta u(t, \cdot)+g(t,\cdot, u(t,\cdot))\, .
\end{align}
We assume there exists a constant $C > 0$ such that 
\begin{equation}\label{eq:Lipschitzness_of_g}
        \norm{g(t,\cdot,v)-g(t,\cdot,w)}_{L^2(\Omega)}\leq C\norm{v-w}_{L^2(\Omega)}\,,
\end{equation}
for all $v, w \in \cU$ and all times $t \in [0,T]$ and we assume that the solution space $\cU$ embeds into $H^2(\Omega)$. Consider the implicit Euler discretization of \eqref{eq:Prelim:PDE} with time-step size $\delta t$ and assume that $1+(\lambda^*-C)\delta t > 0$, where $\lambda^*>0$ is the smallest non-zero eigenvalue of $-\Delta$ with Dirichlet  boundary conditions. Let $e_k(\cdot)$ be the time-integration error such that
\begin{equation}\label{eq:definition_ek_laplacian}
e_k(\cdot)=u(t_{k + 1}, \cdot) - u(t_k, \cdot) -\delta t f(t_{k+1}, \cdot, u(t_{k+1}, \cdot))     ,\ \quad k \in \mathbb{N}\,.    
\end{equation}
Using $\|\cdot\|_M = \|\cdot\|_{L^2(\Omega)}$ in the DtO schemes, 
the error of the DtO solution $\hat{u}(\bt_k, \cdot)$ can be bounded as
\begin{equation*}
        \begin{aligned}
        \|u(t_k,\cdot)-&\hat{u}(\bt_{k}, \cdot)\|_{L^2(\Omega)} \leq \left(\frac{1}{1+(\lambda^*-C)\delta t}\right)^k \norm{u(0,\cdot)-\hat{u}(\bt_{0}, \cdot)}_{L^2(\Omega)} \\
        +& \sum_{i=0}^{k-1} \left(\frac{1}{1+(\lambda^*-C)\delta t}\right)^{k-l}  \left(\norm{r_i^{\rm {DtO}}(\bt_{i+1},\cdot)}_{L^2(\Omega)} + \norm{e_i(\cdot)}_{L^2(\Omega)}\right).
    \end{aligned}   
    \end{equation*}
\end{proposition}
\begin{proof}
  By definition of $e_k(\cdot)$ in \eqref{eq:definition_ek_laplacian} and $r_k^{\text{DtO}}$ in \eqref{eq:Schemes:DtO:ResFun}, we have:
    \begin{multline}
         (1-\delta t\Delta)\big(u(t_{k+1},\cdot)-\hat{u}(\bt_{k+1}, \cdot)\big)=(u(t_k,\cdot) -\hat{u}(\bt_{k}, \cdot)) +\\ \delta t \big(g(t_{k+1}, \cdot, u(t_{k+1}, \cdot))- g(t_{k+1}, \cdot, \hat{u}(\bt_{k+1}, \cdot))\big)  -r_k^{\text{DtO}}(\bt_{k+1},\cdot)+ e_{k}(\cdot) \,.
    \end{multline}
    Multiplying $(I-\delta t\Delta)^{-1}$ on both sides and using the triangle inequality, we obtain
    \begin{equation}
    \begin{aligned}
    &\norm{u(t_{k+1},\cdot)-\hat{u}(\bt_{k+1}, \cdot)}_{L^2(\Omega)}\\
        \leq &\frac{1}{1+\lambda^*\delta t}\left(\norm{(u(t_k,\cdot) -\hat{u}(\bt_{k}, \cdot))}_{L^2(\Omega)}+\norm{r_k^{\text{DtO}}(\bt_{k+1},\cdot)}_{L^2(\Omega)}+ \norm{e_{k}(\cdot)}_{L^2(\Omega)}\right)\\ &+\frac{1}{1+\lambda^*\delta t}\left(\delta t\norm{  g(t_{k+1}, \cdot, u(t_{k+1}, \cdot))- g(t_{k+1}, \cdot, \hat{u}(\bt_{k+1}, \cdot)) }_{L^2(\Omega)}\right) \,.
        \end{aligned}
    \end{equation}
    Using the Lipschitzness of $g (t, \cdot, u(t_k, \cdot))$ in~\eqref{eq:Lipschitzness_of_g} leads to  
    \begin{equation}
    \begin{aligned}    
    &\norm{u(t_{k+1},\cdot)-\hat{u}(\bt_{k+1}, \cdot)}_{L^2(\Omega)} \\
    \leq &\frac{1}{1+(\lambda^*-C)\delta t} \left(\norm{u(t_k,\cdot)-\hat{u}(\bt_{k}, \cdot)}_{L^2(\Omega)}+\|r_k^{\text{DtO}}(\bt_{k+1},\cdot)\|_{L^2(\Omega)}+ \norm{e_k(\cdot)}_{L^2(\Omega)}\right).
    \end{aligned}
    \end{equation}
    Iterating the above inequality leads to the final result.
\end{proof}

The error of solutions of DtO schemes is bounded by a combination of errors in the initial condition, the time discretization error of the PDE, and the norm of the residual at each time step. If one can optimize the residual sufficiently, then the DtO solution error will reduce accordingly. The above analysis also applies to implicit-explicit time integration, namely explicit time integration on $g$ and implicit time integration on $\Delta u$. This could have potential computational advantages because implicit-explicit time integration avoids having to differentiate the nonlinear function $g$ during the optimization of parameter $\bftheta_{k + 1}$ corresponding to the DtO solution.

\subsection{Tangent space collapse: OtD versus DtO schemes} 
In Section~\ref{sec:OtD:TangentSpace}, we discussed that the dynamics in OtD schemes can be affected by the tangent collapse phenomenon. 
For DtO schemes, the situation is different. 
There is no explicit dependency of the a posteriori bounds of DtO solutions on the tangent spaces and thus DtO schemes can be seen as beneficial with respect to the tangent space collapse phenomenon.  
In each time step, an optimization problem is solved to determine the subsequent parameter, rather than the time derivative of the parameter in time. Consequently, in principle, an optimization algorithm can lead to intermediate parameter trajectories that traverse the entire parameter space, and the expressiveness of the nonlinear parametrization remains unaffected by linear dependencies of component functions of the gradient at the current solution. An optimization that explores the whole parameter space, however, can be computationally expensive. In this sense, DtO schemes can trade computational cost for accuracy. 

It is essential to be careful in the selection of optimization algorithms for the residual minimization in DtO schemes because the dependency of DtO schemes on tangent spaces can enter implicitly via the choice of the optimization algorithm that is used to minimize the DtO residual norm. The used optimization methods  should not crucially depend on the tangent space at previous solutions, otherwise similar degeneracy issue caused by tangent space collapse as in OtD schemes will also apply to DtO schemes. In fact, as we will show in Section~\ref{sec:Discussion of OtD and DtO schemes}, DtO schemes can be viewed as first-order approximation of OtD schemes if the Gauss-Newton method is used for optimization. In such cases, the update process relies on the tangent space and thus can still be susceptible to tangent collapse.

\subsection{Stability of DtO schemes}
\label{sec:stability analysis of DtO}
The error analysis in the preceding section depends on a bound of the residual and thus relies on achieving an adequate optimization of the residual. The optimization is in general a delicate task as the underlying optimization problem is typically non-convex. 

In this section, we show that in fact, the stationary points of the DtO residual objective can still be well behaved. More precisely, we show that under mild assumptions, these stationary point solutions are stable, even though we do not have guarantees on their accuracy.

\begin{proposition}
\label{prop-DtO-stability}
Let the right-hand side $f$ of the PDE \eqref{eq:Prelim:PDE} satisfy
\begin{align}\label{eq:f_affine_bound}
\norm{f(t, \cdot, v)}_M\leq C\norm{v}_M + C_0\,,    
\end{align}
for some constants $C, C_0 > 0$ and all $v \in \cU$ and times $t \in [0,T]$. 
Let the nonlinear parametrization satisfy assumption \eqref{eq:Asm:UInT}. For $k \in \mathbb{N}$, let $\hat{u}(\bt_k, \cdot)$ satisfy the stationary point condition \eqref{eq:Schemes:DtO:FirstOrderOpti:zeta_1} corresponding to DtO schemes with explicit Euler time discretization and time-step size $\delta t > 0$. Then, for any $\epsilon$ satisfying $1-\epsilon\delta t>0$, it holds that
 \begin{equation}
        \norm{\hat{u}(\bt_{k}, \cdot)}_M^2\leq \left(\frac{1+2C^2\delta t/\epsilon}{1-\epsilon\delta t}  \right)^k  \norm{\hat{u}(\bt_0, \cdot)}_M^2+  \left(\left(\frac{1+2C^2\delta t/\epsilon}{1-\epsilon\delta t}\right)^k -1\right)\frac{2C_0^2}{2C^2+\epsilon^2}\,.
\end{equation}
\end{proposition}
\begin{proof}
    Stationary points satisfy condition \eqref{eq:Schemes:DtO:FirstOrderOpti:zeta_1}, which we write as
    \begin{equation}
        \left\langle \nabla_\bt \hat{u}(\bt_{k+1}, \cdot), \hat{u}(\bt_{k+1}, \cdot)- \hat{u}(\bt_{k}, \cdot)  \right\rangle_M  = \delta t \left\langle \nabla_\bt \hat{u}(\bt_{k+1}, \cdot),  f(t_k, \cdot, \hat{u}(\bt_{k}, \cdot)) \right\rangle_M. \label{eq:stationary}
    \end{equation}
    By assumption \eqref{eq:Asm:UInT}, we have $\hat{u}(\bt, \cdot) \in T_{\hat{u}(\bftheta,\cdot)}\cM$, so we can represent
    $\hat{u}(\bt_{k+1}, \cdot)$ as a linear combination of functions that span the tangent space $T_{\hat{u}(\bftheta,\cdot)}\cM$, e.g., the component functions of $\nabla_{\bt} \hat{u}(\bt_{k+1}, \cdot)$. Therefore,  we obtain via multiplication with coefficients and summation from \eqref{eq:stationary} that the following equation must hold at a stationary point
    \[
        \left\langle\hat{u}(\bt_{k+1}, \cdot), \hat{u}(\bt_{k+1}, \cdot)- \hat{u}(\bt_{k}, \cdot)  \right\rangle_M  = \delta t \left\langle\hat{u}(\bt_{k+1}, \cdot),  f(t_k, \cdot, \hat{u}(\bt_{k}, \cdot)) \right\rangle_M.
    \]
    Using the Cauchy-Schwarz inequality and the inequality of arithmetic and geometric means, the left-hand side can be lower bounded as
    \[\left\langle\hat{u}(\bt_{k+1}, \cdot), \hat{u}(\bt_{k+1}, \cdot)- \hat{u}(\bt_{k}, \cdot)  \right\rangle_M \geq \frac{1}{2}(\|\hat{u}(\bt_{k+1}, \cdot)\|_M^2 - \|\hat{u}(\bt_{k}, \cdot)\|_M^2).\] 
    Furthermore, using the Cauchy-Schwarz inequality and  bound~\eqref{eq:f_affine_bound}, we obtain
    \begin{equation}
    \label{eqn-stability-DtO-lipschitz-1}
    \begin{aligned}
      \delta t  \left\langle\hat{u}(\bt_{k+1}, \cdot),  f(t_k, \cdot, \hat{u}(\bt_{k}, \cdot)) \right\rangle_M
        \leq &\frac{1}{2}\delta t(\epsilon \norm{\hat{u}(\bt_{k+1})}_M^2+ \frac{\norm{f(t_k, \cdot, \hat{u}(\bt_{k}, \cdot))}_M^2}{\epsilon})\\
        \leq &\frac{1}{2}\delta t(\epsilon \norm{\hat{u}(\bt_{k+1})}_M^2+ \frac{2C^2\norm{\hat{u}(\bt_{k}, \cdot)}_M^2}{\epsilon}+\frac{2C_0^2}{\epsilon}),
        \end{aligned}
    \end{equation} 
    for any $\epsilon > 0$.
    Rearranging the terms in~\eqref{eqn-stability-DtO-lipschitz-1} and leveraging the assumption that  $1-\epsilon\delta t > 0$, we arrive at
    \begin{align}
    \label{eqn-stability-DtO-lipschitz}
        \norm{\hat{u}(\bt_{k+1}, \cdot)}_M^2\leq \frac{1}{1-\epsilon\delta t}  \left(1+\frac{2C^2\delta t}{\epsilon} \right)\norm{\hat{u}(\bt_{k}, \cdot)}_M^2+\frac{1}{1-\epsilon\delta t}\frac{2C_0^2\delta t}{\epsilon}.
    \end{align}
    Iterating the inequality~\eqref{eqn-stability-DtO-lipschitz} yields the final result.
\end{proof}

We can achieve a similar stability bound when $f$ is unbounded, for example when
$f(t, \cdot, u(t,\cdot))=\Delta u+ g(t, \cdot, u(t,\cdot))$ with $\norm{g(t,\cdot, v)}_{L_2(\Omega)}\leq C\norm{v}_{L^2(\Omega)} + C_0$. However, as discussed in Section \ref{sec: Analysis of DtO}, it is more appropriate to employ implicit time integration when dealing with unbounded $f$. We specifically delve into the scenario where we employ implicit time integration for the unbounded term $\Delta u$ while employing explicit time integration for $g$. More precisely, we obtain the following residue function,
\begin{equation}
\label{eqn-residue-explicit-implicit}
    \rdto_k(\bftheta, \cdot) = \hat{u}(\bftheta, \cdot) - \big(\hat{u}(\bftheta_k,\cdot)  + \delta t \Delta \hat{u}(\bftheta, \cdot) +\delta t g(t_{k},\cdot, \hat{u}( \bftheta_k,\cdot))\big) ,
\end{equation}
and the corresponding stationary condition for minimizing the norm $\|\rdto_k(\bftheta, \cdot)\|_{L^2(\Omega)}^2$
\begin{equation}
\label{eqn-stationary-explicit-implicit}
        \left\langle \nabla_\bt \hat{u}(\bt_{k+1}, \cdot), \hat{u}(\bt_{k+1}, \cdot)- \hat{u}(\bt_{k}, \cdot) -\delta t \Delta \hat{u}(\bt_{k+1},\cdot) - \delta t g(t_k, \cdot, \hat{u}(\bt_{k}, \cdot)  \right\rangle_{L^2(\Omega)}  = 0.
    \end{equation}
We have the following stability result for these stationary points.
\begin{proposition}
Let $f(t, \cdot, v)=\Delta u+ g(t, \cdot, v)$ where it holds $\norm{g(t,\cdot, v)}_{L^2(\Omega)}\leq C\norm{v}_{L^2(\Omega)} + C_0$, for some constants $C, C_0 > 0$ and all $v \in \cU$ and times $t \in [0,T]$. 
Let assumption \eqref{eq:Asm:UInT} hold. Consider the solution $\hat{u}(\bt_k, \cdot)$ obtained by DtO which satisfies the stationary point condition \eqref{eqn-stationary-explicit-implicit}. Moreover $\hat{u}(\bt(t), \cdot)$ satisfies homogeneous Dirichlet boundary conditions.  
Then, for any $\epsilon$ satisfying $1-\epsilon\delta t>0$, it holds that
 \begin{equation*}
        \norm{\hat{u}(\bt_{k}, \cdot)}_{L^2(\Omega)}^2\leq \left(\frac{1+2C^2\delta t/\epsilon}{1-\epsilon\delta t}  \right)^k  \norm{\hat{u}(\bt_0, \cdot)}_{L^2(\Omega)}^2+  \left((\frac{1+2C^2\delta t/\epsilon}{1-\epsilon\delta t})^k -1\right)\frac{2C_0^2}{2C^2+\epsilon^2}\,.
\end{equation*}
\end{proposition}
\begin{proof}
    The proof is similar to the proof of Proposition \ref{prop-DtO-stability}. We use Assumption \eqref{eq:Asm:UInT} to replace $\nabla_{\bt} \hat{u}(\bt_{k+1}, \cdot)$ in the \eqref{eqn-stationary-explicit-implicit} by $\hat{u}(\bt_{k+1}, \cdot)$. Then we use the same Cauchy-Schwarz inequality as in Proposition~\ref{prop-DtO-stability} and the additional fact (using the homogeneous boundary condition and integration by parts) that \[\left\langle  \hat{u}(\bt_{k+1}, \cdot), -\delta t \Delta \hat{u}(\bt_{k+1},\cdot)\right\rangle_{L^2(\Omega)} = \delta t \|\nabla \hat{u}(\bt_{k+1}, \cdot)\|_{L^2(\Omega)}^2 \geq 0,  \]
    which will lead to the same estimates as in \eqref{eqn-stability-DtO-lipschitz-1} and  \eqref{eqn-stability-DtO-lipschitz}. Iterating the inequality completes the proof.
\end{proof}

The above proposition implies that if we apply an implicit time integrator on the unbounded Laplacian part while using explicit integrators on other bounded parts, the DtO solution remains stable.

\section{Further discussion on OtD and DtO schemes}
\label{sec:Discussion of OtD and DtO schemes}
We show in Section~\ref{sec-Connections and Differences Between OtD and DtO} that solutions corresponding to OtD and DtO schemes coincide under a specific choice of optimization method and time integrator. 
However, in more general settings, OtD and DtO schemes can behave differently. Section~\ref{sec:OtDGradientFlows} discusses the special case of applying OtD schemes to gradient flows and draws connections to optimization and sampling methods that can be interpreted as OtD schemes.

\subsection{OtD as a first-order approximation of DtO schemes}
\label{sec-Connections and Differences Between OtD and DtO}

The following proposition shows that OtD and DtO schemes follow the same dynamics in the specific case of explicit Euler time intergration and one-step Gauss-Newton optimization. The one-step Gauss-Newton optimization serves the role of providing a first-order approximation of the DtO dynamics.
\begin{proposition}\label{thm:OtD_DtO_OnestepGaussNewton}
Consider the DtO scheme with explicit Euler time integration and time-step size $\delta t > 0$. 
Let $\bftheta_1, \bftheta_2, \dots, \bftheta_K$ be a parameter trajectory that is obtained by applying Gauss-Newton iterations \eqref{eq:Compute:DtO:GaussNewton} of a single step, $L = 1$, to \eqref{eq:Schemes:DtO:RegProb}. 
Then the parameters satisfy the first-order optimality condition 
\eqref{eq:OtD_explicit} 
corresponding to the OtD scheme based on explicit Euler integration and time-step size $\delta t$.
\end{proposition}
\begin{proof}
For the explicit Euler scheme in DtO, we have 
\begin{equation}\label{eq:SecSimilar:rkDtO}
        r_k^{\text{DtO}} (\bftheta, \cdot)= \hat{u}( \bftheta,\cdot)-\big(\hat{u}(\bt_{k},\cdot)+\delta t f(t_k,\cdot, \hat{u}( \bt_{k},\cdot))\big)\, .
    \end{equation}
    Plugging $\bftheta_k$ into \eqref{eq:SecSimilar:rkDtO} provides $r_k^{\text{DtO}} (\bftheta_k,\cdot)=-\delta t f(t_k,\cdot, \hat{u}(\bt_{k},\cdot))$ and $\nabla_\bt  r_k^{\text{DtO}}(\bt,\cdot)=\nabla_\bt \hat{u}(\bt,\cdot)$.
    Plugging these into  the one-step Gauss-Newton method \eqref{eq:Compute:DtO:GaussNewton} with $\bt^{(0)}_{k+1}=\bt_{k}$ and $L=1$, we obtain that the system of equations
    \begin{align}
         \left\langle\nabla_\bt \hat{u}(\bt_{k},\cdot), \nabla_\bt \hat{u}(\bt_{k},\cdot)\right\rangle_M (\bt_{k+1}-\bt_{k}) = \delta t\left\langle\nabla_{\bt }\hat{u}(\bt_{k},\cdot),    f(t,\cdot, \hat{u}(\bt_{k},\cdot)) \right\rangle_M  \, .
         \label{eq:one-stepGaussNewtonDtO}
    \end{align}
    The same condition is obtained when applying the explicit Euler discretization to OtD schemes, which is given in \eqref{eq:OtD_explicit}. 
\end{proof}

If multiple Gauss-Newton iterations are performed, then the OtD and DtO dynamics can become different. Additionally, for other time discretization schemes than explicit Euler, the dynamics of OtD and DtO schemes can be different even with one-step Gauss-Newton optimization. 

\subsection{OtD schemes for gradient flows}\label{sec:OtDGradientFlows}
In this section, we focus on a specific type of evolution equation that can be described as gradient flows in a function space. These equations are widespread in physics \cite{peletier2014variational} and also arise commonly from optimization and sampling algorithms~\cite{ambrosio2005gradient,trillos2023optimization,chen2023sampling}. Under this context, we show the equivalence between the OtD schemes for integrating the gradient flow equations and the natural gradient descent algorithm for optimization within a parametric class.

\subsubsection{Gradient flows on Riemannian manifolds}
Gradient flow equations can be defined on Riemannian manifolds, under the associated geometry. Consider a Riemannian manifold $\mathcal{N}$ of functions with domain $\Omega \subset \bR^d$ and range $\mathbb{R}$. We denote the tangent space at $u \in \mathcal{N}$ as $T_u\mathcal{N}$ and the associated Riemannian metric as $g_u:T_u\mathcal{N}\times T_u\mathcal{N} \to \bR $. The inner product and norm of the tangent space are $\langle\cdot,\cdot\rangle_{g_u}$ and $\|\cdot\|_{g_u}$, respectively. An important special case is when the manifold is a Hilbert space, which we will consider in our examples below. Given an objective function $E:\mathcal{N} \to \bR$ that is continuously differentiable, its Riemannian gradient is denoted by $\nabla_u E \in T_u\mathcal{N}$ with the property
\[\langle \nabla_u E, v\rangle_{g_u} = \lim_{\epsilon \to 0} \frac{E(\gamma(\epsilon)) - E(\gamma(0))}{\epsilon}\, , \]
where $\gamma(\epsilon)$ is any smooth curve on $\mathcal{N}$ satisfying $\gamma(0) = u$ and $\gamma'(0) = v \in T_u\mathcal{N}$.
The gradient flow equation is a time-dependent PDE for functions with domain $\Omega$,
\begin{equation}\label{eq:OtDGrad:ManiFlow}
\partial_t u = -\nabla_u E \, .
\end{equation}

We apply the OtD scheme with the inner product $\langle \cdot, \cdot \rangle_{g_{{u}}}$ to the gradient flow equation \eqref{eq:OtDGrad:ManiFlow} and obtain
\begin{equation*}
\begin{aligned}
     \left\langle\nabla_\bt \hat{u}(\bt(t), \cdot), \nabla_\bt \hat{u}(\bt(t), \cdot)  \right\rangle_{g_{\hat{u}}}\dot{\bt}(t) & =\left\langle\nabla_\bt \hat{u}(\bt(t), \cdot), \nabla_u E|_{u = \hat{u}(\bt(t),\cdot)}(\cdot) \right\rangle_{g_{\hat{u}}}\\
     & = \nabla_{\bftheta} E(\hat{u}(\bt(t), \cdot))\, .
\end{aligned}
\end{equation*}
Then, with time step size $\delta t$, the OtD scheme with the explicit Euler discretization is
\begin{equation}\label{eq:Disc:GradientFlowNGD}
\bt_{k+1} = \bt_{k} - \delta t \bfP_k^{-1}\nabla_{\bftheta} L(\bftheta_k)\, ,
\end{equation}
where we define $L(\bftheta) = E(\hat{u}(\bftheta,\cdot))$ and $\bfP_k = \left\langle\nabla_\bt \hat{u}(\bt_{k},\cdot), \nabla_\bt \hat{u}( \bt_{k},\cdot)\right\rangle_{g_{\hat{u}}}$, which is assumed to be invertible. 
One can see that the dynamics given in \eqref{eq:Disc:GradientFlowNGD} are the same as the ones obtained when applying preconditioned gradient descent to minimizing $L$ over the parameter space $\Theta$. %
As a consequence, applying OtD to gradient flows with the inner product $\langle \cdot, \cdot \rangle_{g_{\hat{u}}}$ and explicit Euler discretizations will lead to the same parameter updates as applying preconditioned gradient descent under the geometry induced by $\langle \cdot, \cdot \rangle_{g_{\hat{u}}}$ over the parameter space $\Theta$. Such preconditioned gradient descent algorithms are also known as natural gradient descent; see \cite{JMLR:v21:17-678}. 
Notice that the inner product $\langle \cdot, \cdot \rangle_{g_{\hat{u}}}$ depends on $\hat{u}(\bftheta, \cdot)$ and thus changes over the iterations $k \in \mathbb{N}$.

\subsubsection{Examples of gradient flows in Hilbert spaces and on manifolds}
Let us first consider the least-squares loss $E: \Ucal \to \mathbb{R}, v \mapsto \frac{1}{2}\|v - g\|_{L^2(\Omega)}^2$, where $v, g$ are functions in the Hilbert space $L^2(\Omega)$ 
equipped with the $L^2(\Omega)$ inner product. 
The gradient is $\nabla E(v) = v - g$ and also a function in $L^2(\Omega)$. Applying OtD as described in the previous paragraph gives \eqref{eq:Disc:GradientFlowNGD} with $\nabla_{\bftheta} E(\hat{u}(\bftheta, \cdot)) = \langle \nabla_{\bftheta}\hat{u}(\bftheta, \cdot), \hat{u}(\bftheta, \cdot) - g\rangle_{L^2(\Omega)}$, which is the gradient $\nabla_{\bftheta} L$ of the least-squares loss 
$L(\bftheta) = \frac{1}{2}\|\hat{u}(\bftheta, \cdot) - g\|_{L^2(\Omega)}^2$ 
over the parameter space $\Theta$. Applying preconditioned gradient descent to the least-squares loss over the parameter space $\Theta$ is equivalent to applying OtD to the corresponding gradient flow over $L^2(\Omega)$.

If $\mathcal{N}$ is the probability density space endowed with the Fisher-Rao metric, the matrix $\bfP_k$ is known as the Fisher information matrix \cite{amari2016information}. 
In fact, using the formula for the Fisher-Rao metric, we have
\begin{equation*}
\begin{aligned}
    \left\langle\nabla_\bt \hat{u}(\bt(t), \cdot), \nabla_\bt \hat{u}(\bt(t), \cdot)  \right\rangle_{g_{\hat{u}}} &= \int \frac{\nabla_\bt \hat{u}(\bt(t), \bx) (\nabla_\bt \hat{u}(\bt(t), \bx))^T}{\hat{u}(\bt(t),\bx)} \, {\rm d}\bx\\
    & = \mathbb{E}_{\bx \sim \hat{u}(\bt(t),\cdot)}[\nabla_\bt \log \hat{u}(\bt(t), \bx) (\nabla_\bt \log  \hat{u}(\bt(t), \bx))^T]\, ,
\end{aligned}
\end{equation*}
where the last term is the definition of the Fisher information matrix. Note that here, $\hat{u}(\bt(t), \cdot)$ is a probability density.
The concept of a natural gradient is first proposed in such context in \cite{amari1998natural}, where $E$ is the loss in maximum likelihood estimates. Therefore, OtD schemes with explicit Euler discretizations can recover this natural gradient descent algorithm when applied to gradient flow equations in the probability density space under the Fisher-Rao metric.

If $\mathcal{N}$ is taken to be the space of quantum wave functions endowed with the Fubini-Study metric, then $\bfP_k$ is related to the Quantum Geometric Tensor \cite{stokes2020quantum}. Furthermore, if $E$ is the Rayleigh quotient for a given Hamiltonian, then the natural gradient descent algorithm is equivalent to stochastic configuration in variational quantum Monte Carlo \cite{becca_sorella_2017}. Our discussion implies that such algorithm can also be recovered by applying the OtD scheme to the gradient flow of $E$ on $\mathcal{N}$.
There are also many other examples, such as the energy natural gradient descent for solving PDEs \cite{muller2023achieving} and in PDE-based optimization \cite{nurbekyan2023efficient} and statistics \cite{chen2020optimal}.

It is worth noting that in the preceding discussions in this section, we consistently use the same inner product $\langle \cdot, \cdot \rangle_{M} = \langle \cdot, \cdot \rangle_{g_{\hat{u}}}$ for the OtD scheme and the gradient, which results in the equivalence between OtD  and the natural gradient descent. However, if the metric $\|\cdot\|_M$ used in the OtD scheme differs from $\|\cdot\|_{g_{\hat{u}}}$, then OtD schemes can give rise to distinct dynamics compared to natural gradient descent. This aspect holds practical significance, as empirical evidence suggests that taking these differences into account may potentially lead to faster convergence.
For instance, in the quantum Monte Carlo algorithm proposed in \cite{neklyudov2023wasserstein}, the gradient flow is defined using the Wasserstein or the Wasserstein-Fisher-Rao metric, while their algorithm is equivalent to OtD schemes with $\|\cdot\|_M$ chosen as the Fisher-Rao metric. In this regard, OtD schemes might offer the potential for generating a wider range of efficient algorithms than natural gradient descent. We leave this as a future avenue for exploring a more systematic design of $\|\cdot\|_M$. Similarly, it would be interesting to explore DtO schemes in these settings too.

\section{Conclusions}
\label{sec: Conclusions}
~While many seemingly different sequential-in-time training methods for nonlinear parametrizations have been developed by various communities for a wide range of problems, this work identifies two broad types of schemes: OtD and DtO schemes. 

The results of this work show that first optimizing and then discretizing in time (OtD) versus first discretizing and then optimizing (DtO) leads to fundamentally different schemes for training nonlinear parametrization sequentially in time. The presented analysis demonstrates that the expressiveness of the tangent spaces of nonlinear parametrization manifolds is key in OtD schemes, whereas the expressiveness of the nonlinear parametrization itself dominates the error of DtO schemes. While the optimization step of OtD schemes is linear if explicit time integration schemes are used, there can be at least numerically a collapse of the tangent space which means dominating residual components can grow unbounded. While DtO schemes circumvent the tangent space collapse phenomenon, they inherently lead to non-convex optimization problems that are challenging to solve numerically to high precision. A perhaps surprising result is that the parameter trajectories of stationary points (rather than optima) in DtO schemes leads to stable dynamics. 

Under the strong assumption of just taking one optimization step with the Gauss-Newton method in DtO schemes, we showed that OtD and DtO solutions coincide, which admits the interpretation that the OtD dynamics correspond to first-order approximations of DtO dynamics, in this special case. The interpretation is also in agreement with the fact that OtD schemes require solving linear least-squares problems over time, even though the parametrizations depend nonlinearly on the parameters. 

Abstractly identifying sequential-in-time methods as either being OtD or DtO schemes paves the way for a better understanding of theoretical and numerical aspects as well as drawing connections between the different methods for leveraging synergies. For example, we showed that a large class of natural gradient descent methods can be described as OtD schemes applied to gradient flows under various metrics. One example of insight gained with this point of view of interpreting such methods as OtD schemes on gradient flows is allowing to separate the metric used for the gradient from the one used for the OtD dynamics, which could hold practical value in developing novel and more efficient algorithms.

The results of this work open several avenues of future research. First, a better understanding of OtD schemes specifically for gradient flows is of interest, as these are important examples as we discussed above. Second, we showed that the optimization step in DtO scheme is challenging while at the same time offering structure that can be exploited to derive more efficient optimization methods that explicitly target DtO schemes. Third, an open research question is connecting OtD and DtO schemes when implicit time integration schemes are applied.

\bibliographystyle{siamplain}
\bibliography{ref}

\end{document}